\documentclass[12pt]{amsart}

\usepackage[centertags]{amsmath}
\usepackage{amsfonts}
\usepackage{amssymb}
\usepackage{amscd}
\usepackage{amsthm}
\usepackage{citesort}
\usepackage{color}
\usepackage[dvips]{graphicx}
\usepackage[cm]{fullpage}

\newtheorem{theorem}{Theorem}[section]
\newtheorem{lemma}[theorem]{Lemma}

\theoremstyle{definition}

\theoremstyle{remark}
\newtheorem{remark}[theorem]{Remark}

\numberwithin{equation}{section}

\DeclareSymbolFont{SY}{U}{psy}{m}{n}
\DeclareMathSymbol{\emptyset}{\mathord}{SY}{'306}

\DeclareMathOperator*{\slim}{s-lim} 
 
 \DeclareMathOperator{\tr}{tr}
\DeclareMathOperator{\Ran}{Ran}

\DeclareMathOperator{\osc}{osc}

\DeclareMathOperator{\Spec}{Spec}

\DeclareMathOperator{\spec}{Spec}

\DeclareMathSymbol{\newtimes}{\mathbin}{SY}{'264}

\begin{document}
%
%
\def\ra{{\sf R}}
\def\je{{\sf N}}
\def\I{\mathbf{I}}
\def\mA{\mathbf{A}}
\def\mM{\mathbf{M}}
\def\mU{\mathbf{U}}
\def\mX{\mathbf{X}}
\def\mV{\mathbf{V}}
\def\mW{\mathbf{W}}
\def\mY{\mathbf{Y}}
\def\mK{\mathbf{K}}
\def\mR{\mathbf{R}}
\def\mQ{\mathbf{Q}}
\def\mT{\mathbf{T}}
\def\mS{\mathbf{S}}
\def\mHm{\mathbf{H}}
\def\mH{\mathsf{A}}
\def\x{\mathcal{X}}
\def\q{\mathcal{Q}}
\def\b{\mathcal{B}}
\def\g{\mathcal{G}}
\def\d{\mathcal{D}}
\def\K{\mathcal{K}}
\def\H{\mathcal{H}}
\def\vp{\mathcal{V}}
\def\wp{\mathcal{W}}
\def\lp{\mathcal{L}}
\def\SS{\mathcal{S}}
\def\y{\mathcal{Y}}
\def\z{\mathcal{Z}}
\def\osc{\mathrm{osc}}
\def\bphi{\mb{phi}}

\newcommand{\uu}[1]{\mathbf{#1}}
\def\calI{{\mathcal I}}
\newcommand{\jmp}[1]{[\![#1]\!]}
\newcommand{\mvl}[1]{\{\!\!\{#1\}\!\!\}}
\def\hh{{\tt h}}
\def\cc{{\tt c}}
\def\pp{{\tt p}}

\renewcommand{\Im}{{\ensuremath{\mathrm{Im\,}}}}
\renewcommand{\Re}{{\ensuremath{\mathrm{Re\,}}}}

\newcommand{\diag}{\mathrm{diag}}

\newcommand{\HH}{\mathbb{H}}
\newcommand{\Q}{\mathbb{Q}}
\newcommand{\R}{\mathbb{R}}
\newcommand{\T}{\mathbb{T}}
\newcommand{\C}{\mathbb{C}}
\newcommand{\Z}{\mathbb{Z}}
\newcommand{\N}{\mathbb{N}}
\newcommand{\bP}{\mathbb{P}}
\newcommand{\PP}{\mathbb{P}}
\newcommand{\1}{\mathbb{I}}
\newcommand{\E}{\mathbb{E}}
\newcommand{\EE}{\mathsf{E}}

\newcommand{\fG}{\mathfrak{g}}
\newcommand{\fQ}{\mathfrak{Q}}
\newcommand{\fH}{\mathfrak{H}}
\newcommand{\fR}{\mathfrak{R}}
\newcommand{\fF}{\mathfrak{F}}
\newcommand{\fB}{\mathfrak{B}}
\newcommand{\fS}{\mathfrak{S}}
\newcommand{\fD}{\mathfrak{D}}
\newcommand{\fE}{\mathfrak{E}}
\newcommand{\fP}{\mathfrak{P}}
\newcommand{\fT}{\mathfrak{T}}
\newcommand{\fC}{\mathfrak{C}}
\newcommand{\fM}{\mathfrak{M}}
\newcommand{\fI}{\mathfrak{I}}
\newcommand{\fL}{\mathfrak{L}}
\newcommand{\fW}{\mathfrak{W}}
\newcommand{\fK}{\mathfrak{K}}
\newcommand{\fA}{\mathfrak{A}}
\newcommand{\fV}{\mathfrak{V}}
\newcommand{\cA}{{\mathcal A}}
\newcommand{\cB}{{\mathcal B}}
\newcommand{\cC}{{\mathcal C}}
\newcommand{\cD}{{\mathcal D}}
\newcommand{\cE}{{\mathcal E}}
\newcommand{\cF}{{\mathcal F}}
\newcommand{\cG}{{\mathcal G}}
\newcommand{\cH}{{\mathcal H}}
\newcommand{\cI}{{\mathcal I}}
\newcommand{\cJ}{{\mathcal J}}
\newcommand{\cK}{{\mathcal K}}
\newcommand{\cL}{{\mathcal L}}
\newcommand{\cM}{{\mathcal M}}
\newcommand{\cN}{{\mathcal N}}
\newcommand{\cO}{{\mathcal O}}
\newcommand{\cP}{{\mathcal P}}
\newcommand{\cQ}{{\mathcal Q}}
\newcommand{\cR}{{\mathcal R}}
\newcommand{\cS}{{\mathcal S}}
\newcommand{\cT}{{\mathcal T}}
\newcommand{\cU}{{\mathcal U}}
\newcommand{\cV}{{\mathcal V}}
\newcommand{\cX}{{\mathcal X}}
\newcommand{\cW}{{\mathcal W}}
\newcommand{\cZ}{{\mathcal Z}}

\newcommand{\sN}{{\mathsf N}}
\newcommand{\sT}{{\mathsf T}}
\newcommand{\sP}{{\mathsf P}}
\newcommand{\sR}{{\mathsf R}}
\newcommand{\sG}{{\mathsf G}}
\newcommand{\sU}{{\mathsf U}}
\newcommand{\sF}{{\mathsf F}}
\newcommand{\sE}{{\mathsf E}}

\newcommand{\bA}{\mathbf{A}}
\newcommand{\bB}{\mathbf{B}}
\newcommand{\bC}{\mathbf{C}}
\newcommand{\bH}{\mathbf{H}}
\newcommand{\bV}{\mathbf{V}}
\newcommand{\bL}{\mathbf{L}}

\newcommand{\D}{\mathbb{D}}
\newcommand{\NN}{\mathbb{N}}
\newcommand{\RR}{\mathbb{R}}

\newcommand{\sD}{\mathsf{D}}
\newcommand{\sL}{\mathsf{L}}
\newcommand{\U}{\mathsf{U}}
\newcommand{\SL}{\mathsf{SL}}

\newcommand{\Res}{\mathrm{Res}}

\newcommand{\supp}{\mathrm{supp}}

\newcommand{\bw}{\mathbf{w}}
\newcommand{\bp}{\mathbf{p}}
\newcommand{\bc}{\mathbf{c}}

\newcommand{\ii}{\mathrm{i}}
\newcommand{\e}{\mathrm{e}}
\newcommand{\ee}{\varepsilon}
\newcommand{\eet}{\varepsilon^\cT}
%
%
\def\qed{Q.E.D.}
%
%
%
%
\newcommand{\norm}[1]{\Vert #1 \Vert}
\newcommand{\normm}[1]{\Vert #1 \Vert}
\newcommand{\ugl}[1]{\left[ #1 \right]}
\newcommand{\sk}[1]{\left( #1 \right)}
\newcommand{\dual}[1]{\left< #1 \right>}
\newcommand{\abs}[1]{\left| #1 \right|}
\newcommand{\abss}[1]{\vert #1 \vert}
\newcommand{\spa}[1]{\mbox{span}\{ #1 \}}
\newcommand{\absf}[2]{\frac{\abss{#1}}{\abss{#2}} }
\newcommand{\ine}[1]{{\mathbf{#1}}}
\newcommand{\mb}[1]{{\mathbf{#1}}}
\newcommand{\conpo}[1]{\stackrel{#1}{\longrightarrow}}
\newcommand{\enorm}[1]{|\!|\!|#1|\!|\!|}
%
%
\def\imag{{\rm i}}
\def\sinbf{\mathsf{sin}}
\def\cosbf{\mathsf{cos}}
\def\eexp{\text{e}}
\def\region{\mathcal{R}}
\def\slim{\text{{\rm s-lim}}}
\def\wlim{\text{{\rm w-lim}}}
\def\tripleb{\mid\!\mid\!\mid}

\author{Stefano Giani}
\address{ School of Mathematical Sciences
University of Nottingham , University Park, Nottingham, NG7 2RD,  United Kingdom }
\email{stefano.giani@nottingham.ac.uk}

\author{Luka Grubi\v{s}i\'{c}}
\address{University of Zagreb, Department of Mathematics, Bijeni\v{c}ka 30, 10000 Zagreb, Croatia}
\email{luka.grubisic@math.hr}

\author{Jeffrey S. Ovall}
\address{University of Kentucky,
Department of Mathematics, Patterson Office Tower 761, Lexington, KY 40506-0027, USA}
\email{jovall@ms.uky.edu}
\thanks{}

\title[Reliable $hp$-adaptive FEM for eigenvalue/vector problems]{Reliable \textit{a-posteriori} error estimators
for $hp$-adaptive finite element approximations of eigenvalue/eigenvector problems}




\keywords{
eigenvalue problems, finite element methods, a posteriori error
estimates, $hp$-apaptivity
}
\subjclass[2000]{Primary: 65N30, Secondary: 65N25, 65N15}

\date{\today}

\dedicatory{}

\begin{abstract}
We present reliable a-posteriori error estimates for $hp$-adaptive
  finite element approximations of eigenvalue/eigenvector
  problems. Starting from our earlier work on $h$ adaptive finite
  element approximations we show a way to obtain reliable and
  efficient a-posteriori estimates in the $hp$-setting. At the core of
  our analysis is the reduction of the problem on the analysis of the
  associated boundary value problem. We start from the analysis of
  Wohlmuth and Melenk and combine this with our a-posteriori
  estimation framework to obtain eigenvalue/eigenvector approximation
  bounds.
\end{abstract}
\maketitle

\section{Introduction}
\label{Intro}
Accurate computation of eigenvalues and eigenvectors of differential operators defined in planar regions has attracted considerable attention recently. A central paper in this body of work is the 2005 contribution of Trefethen and Betcke~\cite{Betcke2005} on computing eigenpairs for the Laplacian on a variety of planar domains, by the method of particular solutions.  This approach produced highly accurate eigenvalues---correct to $13$ or $14$ digits in some cases---but the approach is limited in its application scope to differential operators whose highest order coefficients are constant and lower order coefficients are analytic, see the discussion from \cite{Eisenstat1974}. In particular this means that handling anisotropic diffusion operators is excluded. For further discussion of recent research in this area see \cite{Betcke2007,Barnett2008,Fairweather2009}.

This limitation excludes many interesting eigenvalue model problems for composite materials, such as those which are of interest for methods of nondestructive sensing (cf. \cite{Ammari2009,Ammari2009a}).  Our interest in problems of this sort is motivated by considerations of photonic crystals and related problems, cf. \cite{Ammari2009b,GiGr2011}. Although such problems are not directly addressed in this work, we do consider examples which have jump discontinuities in the coefficients of the highest and lowest order derivatives and therefore capture
some of the computational difficulties which arise in photonic crystal applications.  In \cite{GiGrOv}, we used an $hp$-adaptive discontinuous Galerkin method, with duality-based (goal-oriented) adaptive refinement, to efficiently produce eigenvalue approximations having  at least $10$ correct digits for several model problems, including those with discontinuous coefficients.

Our experience thus far indicates that such $hp$-DG methods provide the most efficient means of computing eigenvalues in the discontinuous-coefficient case in terms of flops-per-correct-digit.  However, the structure of DG-methods is such that only limited results are available on the accuracy of computed eigen\textit{vector} approximations.  This is, in part, due to the difficulty in choosing an appropriate norm for the analysis.  The analytical framework that we have developed elsewhere for lower-order continuous elements (\cite{Grubisic2009,Bank2010}) uses native Hilbert space norms in an essential way, so standard DG norms appear very difficult to incorporate in this kind of analysis.  

Because it is straight-forward to apply our framework in the analysis of approximations of eigenvectors of low regularity, $H^{1+\alpha}$ for some (small) $\alpha>0$, as well as invariant subspaces corresponding to degenerate eigenvalues (those having multiplicity greater than one), it seems useful to develop a continuous $hp$-adaptive scheme based on this approach.  The aim is that a more robust theory might soon be complemented with computational efficiency which is competitive with its DG counterpart.  The present work is a first significant step in that direction.

The rest of this paper is organized as follows.  In Section~\ref{Model} we introduce our model problem and basic notation related to its continuous and discrete versions, as well as some basic theory related to such eigenvalue problems.  The notion of approximation defects and their relevance is discussed in Section~\ref{ApproxDef}, with results from~\cite{Grubisic2009,Bank2010} extended for use in the present context.  These extensions make possible the incorporation of results from~\cite[Section 3]{Melenk2001}, which pertain to boundary value problem error estimation, to obtain efficient and reliable estimates of eigenvalue approximations, which is discussed in Section~\ref{HP_Adapt}.  Also in this section we present a $\sin\Theta$ type result for the accuracy of eigenvectors---to assess the accuracy of the angle operator we use the Hilbert-Schmidt norm.  Section~\ref{Exper}, which constitutes the bulk of the paper, is devoted to numerical experiments on a variety of different kinds of
problems to assess the practical behavior of the proposed approach.

\section{Model Problem and Discretization}
\label{Model}
Let $\Omega\subset\R^2$ be a bounded polygonal domain,
possibly with re-entrant corners, and let $\partial\Omega_D\subset\partial\Omega$ have positive
(1D) Lebesgue measure. We define the space
$\cH=\{v\in H^1(\Omega):\; v_{|_{\partial\Omega_D}}=0\}$, where these
boundary values are understood in the sense of trace.
We are interested in the eigenvalue problem:
\begin{align}\label{eigprob}
\mbox{Find }(\lambda,\psi)\in\RR\times \cH\mbox{ so that }
B(\psi,v)=\lambda (\psi,v)\mbox{ and } \psi\ne0\mbox{ for all } v\in\cH~.
\end{align}
Here we have assumed
\begin{equation}\label{bil_form}
B(w,v)=\int_\Omega A\nabla w\cdot\nabla v+c wv\,dx,
\end{equation} and
\begin{equation}\label{scalar_product}
(w,v)=\int_\Omega wv\,dx
\end{equation}
is the standard $L^2$ inner-product.  We will also assume that the
diffusion matrix $A$ is piecewise constant and uniformly positive
definite a.e., the scalar $c$ is also piecewise constant and
non-negative.  These assumptions are sufficient to guarantee that
there are constants $\beta_0,\beta_1>0$ such that $B(v,w)\leq
\beta_1\|v\|_1\|w\|_1$ and $B[v]\doteq B(v,v)\geq \beta_0\|v\|_1^2$ for
all $v,w\in\cH$.  In other words $B(\cdot,\cdot)$ is an inner product
on $\cH$, whose induced ``energy''-norm $\enorm{\cdot}$ is equivalent
to $\|\cdot\|_1$.  The numbers $\beta_0$ and $\beta_1$ are referred to, respectively,
as the coercivity and continuity constants for $B$.

Here and elsewhere, we use the following standard notation for norms
and seminorms: for $k\in\N$ and $S\subset\Omega$ we denote the standard norms
and semi-norms on the Hilbert spaces $H^k(S)$ by
\begin{align}\label{norms}
&\norm{v}_{k,S}^2=\sum_{|\alpha|\leq k}\norm{D^\alpha v}^2_{S}
&|v|_{k,S}^2=\sum_{|\alpha|= k}\norm{D^\alpha v}^2_{S}~,
\end{align}
where $\|\cdot\|_{S}$ denotes the $L^2$ norm on $S$.
Additionally, we define $\enorm{\cdot}_S$ by
\begin{align}\label{enorms}
\enorm{v}_{S}^2=\int_S A\nabla v\cdot\nabla v+c v^2\,dx~,
\end{align}
recognizing that this may be a semi-norm.
When $S=\Omega$ we omit it from the subscript.
Our assumptions on
$A$ and $c$ guarantee that there are local constants $\beta_{0S},\gamma_{1S}>0$
such that $\beta_{0S}|v|_{1,s}^2\leq\enorm{v}_S^2\leq \beta_{1S}\|v\|_{1,s}^2$, and
the seminorm in the lower bound can be replaced with the full norm
(after modifying $\beta_{0S}$ if necessary) if $c(x)\geq c_S>0$ on $S$.

\subsection{Notational conventions for eigenvalues and eigenvectors}
The variational eigenvalue problem
(\ref{eigprob})--(\ref{scalar_product}) is attained by the positive
sequence of eigenvalues
\begin{equation}\label{enum}0<\lambda_1\leq\lambda_2\leq\cdots\leq\lambda_q\leq\cdots
\end{equation}
and the sequence of eigenvectors $(\psi_i)_{i\in\N}$ such that
\begin{align}\label{spec_prob}
  B(\psi_i, v)&=\lambda_i(\psi_i, v),\qquad\forall v\in\H,\qquad
  \text{and }(\psi_i, \psi_j)=\delta_{ij}~.
\end{align}
Here we have counted the eigenvalues according to their multiplicity
and we will also use the notation $\psi_i\perp \psi_j$ when $(\psi_i,
\psi_j)=0$ (when $i\neq j$).
Furthermore, the sequence $(\lambda_i)_{i\in\N}$ has no finite
accumulation point; and due to the Peron-Frobenius theorem we know
that, in the case in which $\Omega$ is a path-wise connected domain,
the inequality $\lambda_1<\lambda_2$ holds and the eigenvector
$\psi_1$ can be chosen so that $\psi_1$ is continuous and $\psi_1>0$
holds pointwise in $\Omega$. We will also use the notation
$$
\Spec_B:=\{\lambda_i~:~i\in\N\}
$$
to denote the spectrum of the variational eigenvalue problem (\ref{spec_prob}) and we use
$$
\mathfrak{M}(\lambda):=\text{span}\{\psi~:~\|\psi\|=1, \text{ and }\; B(\psi, \phi)=\lambda(\psi,\phi),\quad\forall\phi\in\cH\}
$$
to denote the spectral subspace associated to $\lambda\in\Spec_B$. For variational eigenvalue problems like (\ref{bil_form}) and
(\ref{spec_prob}) the subspaces $\mathfrak{M}(\lambda)$, $\lambda\in\Spec_B$ are finite dimensional.
Furthermore, let $E_\lambda$ be the $L^2$ orthogonal projection onto $\mathfrak{M}(\lambda)$ then
\begin{align*}
\sum_{\lambda\in\Spec_B}E_\lambda=I
\end{align*}
and the spaces $\mathfrak{M}(\lambda)=\Ran E_\lambda$ and $\mathfrak{M}(\mu)=\Ran E_\mu$ are
mutually orthogonal for $\lambda,\mu\in\Spec_B$ and $\lambda\ne\mu$.

Finally, note that
$$
B(\psi, \phi)=\sum_{\lambda\in\spec(\mH)}\lambda(\psi, E_\lambda\phi),\qquad \psi,\phi\in\H
$$
and so we obtain an alternative representation of the energy norm
\begin{equation}\label{dichotomy}
\enorm{\psi}^2=B(\psi, \psi)=\sum_{\lambda\in\spec(\mH)}\lambda(\psi, E_\lambda\psi),\quad\psi\in\cH.
\end{equation}

\subsection{Discrete eigenvalue/eigenvector approximations}
We discretize~\eqref{eigprob} using $hp$-finite element spaces, which
we now briefly describe.  Let $\cT=\cT_h$ be a triangulation of
$\Omega$ with the piecewise-constant mesh function $h:\cT_h\to(0,1)$,
$h(K)=\text{diam}(K)$ for $K\in\cT_h$.  Throughout we
implicitly assume that the mesh is aligned with all discontinuities of
the data $A$ and $c$, as well as any locations where the (homogeneous)
boundary conditions change between Dirichlet and Neumann.  Given a
piecewise-constant distribution of polynomial degrees, $p:\cT_h\to\N$,
we define the space
\begin{align*}
V=V_h^p=\{v\in\cH\cap C(\overline\Omega):\;
v\big|_K\in\PP_{p(K)}\mbox{ for each }K\in\cT_h\}~,
\end{align*}
where $\PP_j$ is the collection of polynomials of total degree no
greater than $j$ on a given set.  Suppressing the mesh parameter $h$
for convenience, we also define the set of edges $\cE$ in $\cT$, and
distinguish interior edges $\cE_I$, and edges on the Neumann boundary $\cE_N$
(if there are any).  Additionally, we let $\cT(e)$ denote the one or
two triangles having $e\in\cE$ as an edge, and we extend $p$ to $\cE$ by
$p(e)=\max_{K\in\cT(e)}p(K)$.  As is standard, we assume that the
family of spaces satisfy the following regularity properties on
$\cT_h$ and $p$:
There is a constant $\gamma>0$ for which
\begin{enumerate}
\item[(C1)] $\gamma^{-1}[h(K)]^2\leq \mbox{area}(K)$
for  $K\in\cT$,
\item[(C2)] $\gamma^{-1}(p(K)+1)\leq p(K')+1\leq\gamma(p(K)+1)$
for adjacent $K,K'\in\cT$, $\overline{K}\cap\overline{K'}\ne\emptyset$.
\end{enumerate}
It is really just a matter of notational convenience that a single
constant $\gamma$ is used for all of these upper and lower bounds.
The shape regularity assumption (C1) implies that the diameters of
adjacent elements are comparable.

In what follows we consider the discrete versions of~\eqref{eigprob}:
\begin{align}\label{disceigprob}
\mbox{Find } (\hat\lambda,\hat\psi)\in\RR\times V
\mbox{ such that } B(\hat\psi,v)=\hat\lambda (\hat\psi,v)\mbox{ for all } v\in V~.
\end{align}
We also assume, without further comment, that the solutions are
ordered and indexed as in \eqref{enum}, with
$(\hat\psi_i,\hat\psi_j)=\delta_{ij}$.  We are interested in assessing
approximation errors in collections of computed eigenvalues and
associated invariant subspaces.
Let $s_m=\{\mu_k\}_{k=1}^m\subset(a,b)$ be the set of all eigenvalues
of $B$, counting multiplicities, in the interval $(a,b)$, and let
$S_m=\mathrm{span}\{\phi_k\}_{k=1}^m$ be the associated invariant
subspace, with $(\phi_i,\phi_j)=\delta_{ij}$.  The discrete
problem~\eqref{disceigprob}
is used to compute corresponding approximations
$\hat{s}_m=\{\hat\mu_k\}_{k=1}^m$ and
$\hat{S}_m=\mathrm{span}\{\hat\phi_k\}_{k=1}^m$, with
$(\hat\phi_i,\hat\phi_j)=\delta_{ij}$.
\begin{remark}
When $s_m$
consists of the smallest $m$ eigenvalues, we use the absolute
labelling $s_m=\{\lambda_k\}_{k=1}^m$ and
$S_m=\mathrm{span}\{\psi_k\}_{k=1}^m$ instead of the relative
labelling involving $(\mu_k,\phi_k)$; and the analogous statement
holds for the discrete approximations $\hat{s}_m$ and $\hat{S}_m$.
This distinction is used in
some of our results, such as Theorems~\ref{tm:ess} and~\ref{thm:asim2}.
\end{remark}


\section{Approximation Defects}
\label{ApproxDef}
\subsection{Approximation defects}
Let the finite element space $V\subset \cH$ be given and let
$\hat{s}_m$ and $\hat{S}_m$ be the approximations which are computed
from $V$. We define the corresponding \textit{approximation defects}
as:
\begin{align}\label{appdefect1.1}
  \eta_{i}^2(\hat{S}_m)=\max_{ \substack{\mathcal{S}\subset
      \hat{S}_{m}\\ \dim\mathcal{S}=m-i+1}}
  \min_{\substack{f\in\mathcal{S}\\f\neq 0}}
  \frac{\enorm{u(f)-\hat{u}(f)}^2}{\enorm{u(f)}^2} ~,
\end{align}
where $u(f)$ and $\hat{u}(f)$ satisfy the boundary value problems:
\begin{align}\label{sol:Hsp}
B(u(f),v)&=(f,v)\mbox{ for every }v\in\cH\\
B(\hat{u}(f),v)&=(f,v)\mbox{ for every }v\in V~.\label{sol:Vsp}
\end{align}
In Theorems~\ref{tm:ess} and~\ref{thm:asim2} below, we state key
theorems from~\cite{Grubisic2009,Bank2010}, which show that these
approximation defects would yield ideal error estimates for eigenvalue
and eigenvector computation \textbf{if they could be computed}.  This
motivates the use of \textit{a posteriori} error estimation techniques
for boundary value problems to efficiently and reliably estimate
approximation defects.  In~\cite{Grubisic2009,Bank2010}, we used
hierarchical bases to estimate the approximation defects when $V$ was
the space of continuous, piecewise affine functions.  In Section
\ref{HP_Adapt} we show how to utilize the theory of residual based
estimates for $hp$-finite elements from \cite{Melenk2001} in a similar
fashion.

The following result concerns approximations $\hat{s}_{\textsc{m}}$ and
$\hat{S}_{\textsc{m}}$ of the (complete) lower part of the spectrum. This is the reason why
we have capitalized the dimension parameter $\textsc{m}\in\N$, which is associated to the cluster of lowermost eigenvalues.
As opposed to a given cluster of eigenvalues contained in the interval $\big(a, b\big)$.

\begin{theorem}\label{tm:ess}
  Assume that $\lambda_{\textsc{m}}<\lambda_{\textsc{m}+1}$, and
  let $\hat{S}_\textsc{m}$ be the span of first $\textsc{m}\in\N$ eigenvectors of~\eqref{disceigprob}.  If
  $\hat{S}_\textsc{m}={\rm
    span}\{\hat\psi_1,\cdots,\hat\psi_\textsc{M}\}$ is such that
  $\frac{\eta_\textsc{m}(\hat{S}_\textsc{m})}{1-\eta_\textsc{m}(\hat{S}_\textsc{m})}<
  \frac{\lambda_{\textsc{m}+1}-\hat\lambda_\textsc{m}}{\lambda_{\textsc{m}+1}+\hat\lambda_\textsc{m}}$
  then
\begin{equation}\label{rely}
  \frac{\hat\lambda_1}{2\hat\lambda_\textsc{m}}
\sum_{i=1}^\textsc{m}\eta_i^2(\hat{S}_\textsc{m})
  \leq\sum^\textsc{m}_{i=1}\!
  \frac{\hat\lambda_i-\lambda_i}
  {\hat\lambda_i}\leq
    C_{\textsc{M}} \sum_{i=1}^\textsc{m}\eta_i^2(\hat{S}_\textsc{m}).
  \end{equation}The constant $C_{\text{M}}$ depends solely on the relative
  distance to the unwanted component of the spectrum
  (e.g. $\frac{\lambda_{\textsc{M}}-\lambda_{\textsc{M}+1}}{
    \lambda_{\textsc{M}}+\lambda_{\textsc{M}+1}}$).
\end{theorem}
The constant $C_{\textsc{M}}$ is given by an explicit formula
which is a reasonable practical overestimate, see \cite{Grubisic2009,Bank2010}
for details.  A similar results holds for the eigenvectors.  We point
the interested reader to \cite[Theorem 4.1 and equation
(3.10)]{Grubisic2009} and \cite[Theorem 3.10]{Bank2010}.
\begin{remark}\label{loweigdegen}
  Although $\lambda_1<\lambda_2$ for the particular problems we
  consider numerically in the present work, much of the theory carries
  over to problems where $\Omega$ is not pathwise connected, or the
  boundary conditions are periodic (as examples).  In these cases the
  Peron-Frobenius theorem does not apply, and it is quite possible
  that the smallest eigenvalue is degenerate.  If this is the case,
  and $\lambda_1=\lambda_\textsc{m}$, then the constant
  ${\hat\lambda_1}/{2\hat\lambda_\textsc{m}}$ in~\eqref{rely} can be
  replaced by $1$.
\end{remark}

An important feature of these ideal estimates is that they are
asymptotically exact, both as eigenvector as well as as eigenvalue
estimators, as the following theorem indicates in the case of a single
degenerate eigenvalue and its corresponding invariant subspace.
\begin{theorem}\label{thm:asim2}
Let $\lambda_q$ be a degenerate eigenvalue of multiplicty $m$,
$\lambda_{q-1}<\lambda_{q}=\lambda_{q+m-1}<\lambda_{q+m}$.
Let $\hat{S}_m=\hat{S}_m(\cT)=\mathrm{span}(\hat\phi_k)\subset V=V(\cT)$
 be the computed approximation of the
invariant subspace corresponding to $\lambda_q$.
Then, taking the pairing of eigenvectors $\phi_i$ and
Ritz vectors $\hat\phi_i$ as in~\cite{Grubisic2009}, we have
\begin{align}\label{eq:asympt_1}
  \lim_{h\to
    0}\frac{\sum_{i=1}^{m}\frac{|\hat\mu_i-\lambda_q|}{\hat\mu_i}}{
\sum_{i=1}^{m}\eta_i^2(\hat{S}_m)}=1\quad,\quad
    \lim_{h\to
    0}\frac{\sum_{i=1}^{m}\frac{B[ \hat\phi_i-
    \phi_{i}]}{B[ \phi_{i}]}}
    {\sum_{i=1}^{m}\eta_i^2(\hat{S}_m)}=1~.
\end{align}
\end{theorem}

\subsection{A relationship with the residual estimates for a Ritz vector basis}

This section addresses the issue of the computability of
$\eta_i(\hat{S}_m)$ by relating these quantities to the
standard dual energy norm estimates of the residuals associated to the Ritz
vector basis of $\hat{S}_m$.

In our notation the energy
norm was denoted by $\enorm{\cdot}$ and we use $u(\cdot)$ and $\hat u(\cdot)$ to denote the solution operators
from (\ref{sol:Hsp}) and (\ref{sol:Vsp}). We assume $\hat \phi_1,
\ldots, \hat \phi_m$ are the Ritz vectors from $\hat{S}_m$, then for
$i,j =1, \ldots, m$, we define the matrices
\begin{align}\label{Gram}
E_{ij}&=B\big(u(\hat\phi_i)-\hat u(\hat\phi_i), u(\hat\phi_j)-\hat u(\hat\phi_j)\big)\\
 G_{ij}&=B\big(u(\hat\phi_i), u(\hat\phi_j)\big).
\end{align}

These matrices were introduced in \cite{Grubisic2009} under the name of the error and the gradient matrix. It was shown in \cite{Grubisic2009} that
$\eta_i(\hat S_m)=\lambda_i(E, G)$, where $\lambda_1(E, G)\leq\cdots\leq\lambda_m(E, G)$ are the eigenvalues of
the generalized eigenproblem for the matrix pair $(E, G)$.

We further assume that $\hat\phi_i$, $i=1,\cdots,m$ are among the Ritz vectors from the finite element subspace $V$, $V\supset\hat S_m$ from (\ref{sol:Vsp}).
The identity (\ref{Gram}) implies that $E$ is a Gram matrix for the set of vectors $u(\hat\phi_i)-\hat u(\hat\phi_i)$, $i=1,\ldots, m$. If
we assume that $\hat{S}_m$ does not contain any eigenvectors, then we conclude that $E$ must be positive definite
matrix. Furthermore, it always holds
\begin{align}
\nonumber \eta_i^2(\hat{S}_m) &=\lambda_i( G^{-1/2}E G^{-1/2})\\\label{diagonal}
E_{ii}&=\mu_i^{-2}\enorm{u(\hat\mu_i\hat\phi_i)-\hat u(\hat\mu_i\hat\phi_i)}^2,\quad i=1, \ldots, m\\
\nonumber D_\mu&\leq G\leq(1+\fD_l)D_\mu,
\end{align}
where $D_\mu=\diag(\hat\mu_1^{-1}, \ldots, \hat\mu_m^{-1})$ and $
\fD_l=\|D_\mu^{-1/2}( G-D_\mu)D_\mu^{-1/2}\|$. Let us note that $\fD_l$ is the relative estimate, so it is expected that even for very crude
finite element spaces $V$ we have $\fD_l<1$.

 Now compute
$$ \sum_{i=1}^m\lambda_i(D_\mu^{-1/2}E
D_\mu^{-1/2})=\tr(D_\mu^{-1/2}E D_\mu^{-1/2})=\sum_{i=1}^mE_{ii}\hat\mu_i,
$$
and so conclude that
\begin{align}\label{eq:preDiag}
\frac{1}{1+\fD_l}\sum_{i=1}^m
E_{ii}\hat\mu_i\leq\sum_{i=1}^m\eta^2_i(\hat{S}_m) &\leq\sum_{i=1}^m
E_{ii}\hat\mu_i
~.
\end{align}
We summarize this considerations --- using the identity (\ref{diagonal}) --- in the following lemma.
\begin{lemma}\label{trace_est}  It holds that
\begin{align}\label{eq:preDiagPr3}
\frac{1}{1+\fD_l}\sum_{i=1}^m \hat\mu_i^{-1}\enorm{u(\hat\mu_i\hat\phi_i)-\hat u(\hat\mu_i\hat\phi_i)}^2\leq\sum_{i=1}^m\eta^2_i(\hat{S}_m) \leq \sum_{i=1}^m \hat\mu_i^{-1}\enorm{u(\hat\mu_i\hat\phi_i)-\hat u(\hat\mu_i\hat\phi_i)}^2~.
\end{align}
\end{lemma}

\section{$hp$-Error Estimation and Adaptivity in the Eigenvalue
  Context}
\label{HP_Adapt}
Using Lemma~\ref{trace_est}, we have reduced the problem of
estimating the approximation defects, and hence the
error in our eigenvalue/eigenvector computations, to that
of estimating error in associated boundary value problems.
In particular, we must estimate $\enorm{u(\hat\mu_i\hat\phi_i)-\hat u(\hat\mu_i\hat\phi_i)}^2$
for each Ritz vector, where $\hat{S}_m=\mbox{span}\{\hat\phi_1,\ldots,\hat\phi_m\}$ is our approximation
of $S_m=\mbox{span}\{\phi_1,\ldots,\phi_m\}$.
We modify key results from~\cite{Melenk2001}, which were stated only for the
Laplacian, to our context.  The identity
$\hat u(\hat\mu_i\hat\phi_i)=\hat\phi_i$, makes our job easier.
We define the element
residuals $R_i$ for $K\in \cT$, and the edge (jump) residuals $r_i$ for $e\in \cE$, by
\begin{align}\label{two_residuals}
{R_i}{|_K}&=\hat\mu_i\hat\phi_i-c \hat\phi_i+\nabla\cdot A\nabla
\hat\phi_i~,\\
{r_i}_{|_e}&=
\begin{cases}
-(A\nabla \hat\phi_i)_{|_K}\cdot\mb{n}_K-(A\nabla\hat\phi_i)_{|_{K'}}\cdot\mb{n}_{K'}&,\, e\in\cE_I\\
-(A\nabla \hat\phi_i)_{|_K}\cdot\mb{n}_K&,\, e\in\cE_N
\end{cases}~.
\end{align}
For interior edges $e\in\cE_I$, $K$ and $K'$ are the two adjacent
elements, having outward unit normals $\mb{n}_K$ and $\mb{n}_{K'}$,
respectively; and for Neumann boundary edges $e\in\cE_N$ (if there are
any), $K$ is the single adjacent element, having outward unit normal
$\mb{n}_K$.  We note that $R$ is a polynomial of degree no greater than
$p(K)$ on $K$, and $r$ is a polynomial of degree no greater than
$p(e)$ on $e$.

Our estimate of $\varepsilon_i^2=\sum_{K\in\cT}\varepsilon_i^2(K)\approx\enorm{u(\hat\mu_i\hat\phi_i)-\hat u(\hat\mu_i\hat\phi_i)}^2$ is computed from local quantities,
\begin{align}\label{vareps_i}
\varepsilon_i^2(K)=\left(\frac{h(K)}{p(K)}\right)^2\|R_i\|_{0,K}^2
+\frac{1}{2}\sum_{e\in \cE_I(K)}\frac{h(e)}{p(e)}\|r_i\|_{0,e}^2+
\sum_{e\in \cE_N(K)}\frac{h(e)}{p(e)}\|r_i\|_{0,e}^2~,
\end{align}
where $\cE_I(K)$ and $\cE_N(K)$ denote the interior edges and Neumann
boundary edges of $K$, respectively.  An inspection the proof
of~\cite[Lemma 3.1]{Melenk2001} (which was stated for the Laplacian)
makes the following assertion clear.
\begin{lemma}\label{reliable} There is a constant $C>0$ depending only on the $hp$-constant $\gamma$ and the coercivity constant $\beta_0$, such that $ \enorm{u(\hat\mu_i\hat\phi_i)-\hat u(\hat\mu_i\hat\phi_i)}^2 \leq C \varepsilon_i^2$.
\end{lemma}
\noindent
A few remarks are in order concerning the lemma above and how it relates to~\cite[Lemma 3.1]{Melenk2001}.  First, the bound in~\cite[Lemma 3.1]{Melenk2001}
includes an additional term involving the difference between the righthand
side (in our case $\phi_i$) and its projection on $K$ into a space of
polynomials.  This additional term only arises in their result because they
have chosen to use the projection of the righthand side, instead of the
righthand side itself, to define the element residual (here called $R_i)$.
They do this in order to employ certain polynomial inverse estimates, which
hold in our case outright because our righthand sides are piecewise polynomial.
Their result also involves a parameter $\alpha\in[0,1]$, which we have
taken to be $0$.  The result~\cite[Lemma 3.1]{Melenk2001} is based
on Scott-Zhang type quasi-interpolation, which naturally gives rise to
errors measured in $H^1$.  Mimicking their arguments with our indicator,
one would arrive at a result of the form
\[
\enorm{u(\hat\mu_i\hat\phi_i)-\hat u(\hat\mu_i\hat\phi_i)}^2 \leq  \tilde{C} \varepsilon_i
\|u(\hat\mu_i\hat\phi_i)-\hat u(\hat\mu_i\hat\phi_i)\|_1~,
\]
where $\tilde{C}$ depends only on $\gamma$.  The constant in the coercivity bound
$\beta_0\|v\|_1^2\leq\enorm{v}^2$ enters Lemma~\ref{reliable} at this final
stage.
Similarly, a careful reading of the proofs of~\cite[Lemma 3.4 and 3.5]{Melenk2001} show that their efficiency results are readily
extended to elliptic operators of the type considered here.
\begin{lemma}\label{efficiency} For any $\epsilon>0$,
there is a constant $c=c(\epsilon)>0$ depending only on the $hp$-constant $\gamma$ and the global continuity constant $\beta_1$, such that
$\varepsilon_i^2(K) \leq c p_K^{2+2\epsilon}\enorm{u(\hat\mu_i\hat\phi_i)-\hat u(\hat\mu_i\hat\phi_i)}_{\omega_K}^2$.
\end{lemma}
\noindent
Here, $\omega_K$ is the patch of elements which share an edge with $K$.  The global continuity constant $\beta_1$ could be replaced in
Lemma~\ref{efficiency} by a local continuity constant $\beta_{1\omega_K}$ if desired.
\begin{remark}\label{degradation}
The $p$-dependence in local efficiency bound of Lemma~\ref{efficiency}
is unfortunately unavoidable in the proof, and would suggest decreased
efficiency of the estimator as $p_K$ is increased if this estimate
were sharp.  Our numerical experiments do seem to indicate that
efficiency does, in fact, decrease under $hp$-refinement, but that
this decrease is modest in practical computations.
\end{remark}

 With these results we now state the main theorem.
\begin{theorem}\label{tm:HP_eta}
  Under the assumptions of Theorem~\ref{tm:ess}, we have the following
  upper- and lower-bounds on eigenvalue error,
\begin{equation}\label{relyhp}
C_{1} \sum_{i=1}^\textsc{m}\hat\lambda_i^{-1}\varepsilon_i^2\leq
\sum^\textsc{m}_{i=1}\!
  \frac{\hat\lambda_i-\lambda_i}
  {\hat\lambda_i}\leq
    C_{2} \sum_{i=1}^\textsc{m}\hat\lambda_i^{-1}\varepsilon_i^2~.
 \end{equation}
 The constant $C_1$ depends solely on the ratio $\hat\lambda_1/(2
 \hat\lambda_2)$, the $hp$-regularity constant $\gamma$, the
 continuity constant $\beta_1$, and the maximal polynomial degree
 $\bar{p}=\max_{K\in\cT}p(K)$.  The constant $C_2$ depends solely on
 the relative distance to the unwanted component of the spectrum, the
 $hp$-regularity constant $\gamma$ and the coercivity constant $\beta_0$.
\end{theorem}
\begin{proof} These assertions follow directly from Theorem
  \ref{tm:ess}, Lemma~\ref{trace_est}, and Lemmas~\ref{reliable}
  and~\ref{efficiency}.
\end{proof}
\begin{remark}\label{local_ind}
It is relative local indicators $\hat\mu_i^{-1}\varepsilon_i^2(K)$
which will be used to mark elements for refinement, as will be
described in Section~\ref{Exper}.
\end{remark}
A similar result holds for the eigenvectors and eigenspaces. Let now $$E(\lambda_\textsc{m})=\sum_{
\lambda\leq\lambda_{\textsc{m}},\;\lambda\in\Spec_B}E_\lambda$$ be the orthogonal projection
onto the eigenspace belonging to the first $\textsc{m}$ eigenvalues of the form $B$ as given in Theorem \ref{tm:HP_eta}. Let also $\|\cdot\|_{S_2}$ be
the Hilbert-Schmidt norm on the ideal of all Hilbert-Schmidt operators, see \cite{SimonTrace}. We now have the eigenvector result.

\begin{theorem}\label{tm:HP_eta_v}
  Under the assumptions of Theorem~\ref{tm:ess}, we have the following
  upper- and lower-bounds on eigenfunction error,
\begin{equation}\label{relyhp_v}
 \|\sin\Theta(E(\lambda_{\textsc{m}}),\hat S_{\textsc{m}})\|_{S_2}\leq
    C_{\textsc{m},\cT} \sqrt{\sum_{i=1}^\textsc{m}\hat\lambda_i^{-1}\varepsilon_i^2}.
  \end{equation}
  The constant $C_{\textsc{m},\cT}$ depends solely on the relative
  distance to the unwanted component of the spectrum
  (e.g. $\frac{\lambda_{\textsc{M}}-\lambda_{\textsc{M}+1}}{
    \lambda_{\textsc{M}}+\lambda_{\textsc{M}+1}}$), the
  $hp$-regularity constant $\gamma$ and the continuity constant
  $\beta_1$.
\end{theorem}

\section{Experiments}
\label{Exper}
In the numerical experiments we illustrate the efficiency of the
estimator~\eqref{relyhp} on
several problems of the general form
\begin{align}\label{gen_eig_prob}
\cL \psi=\lambda\psi\mbox{ in }\Omega\quad,\quad \|\psi\|=1~,
\end{align}
for a second-order, linear elliptic operator $\cL$, where homogeneous
Dirichlet or Neumann conditions are imposed on the boundary.  Plots
are given of the total relative error, its \textit{a posteriori}
estimate, and the associated effectivity quotient, shown,
respectively, below:
\begin{align*}
  \sum^\textsc{m}_{i=1}\!  \frac{\hat\lambda_i-\lambda_i}
  {\hat\lambda_i} \quad,\quad
  \sum_{i=1}^\textsc{m}\hat\lambda_i^{-1}\varepsilon_i^2 \quad,\quad
  \frac{\sum^\textsc{m}_{i=1}\!  \frac{\hat\lambda_i-\lambda_i}
    {\hat\lambda_i}}{\sum_{i=1}^\textsc{m}\hat\lambda_i^{-1}\varepsilon_i^2}~.
\end{align*}
These are plotted against the square-root of the size of the discrete
problem $\mathrm{DOFs}=\text{dim}(V^p_k)$.  For most of our problems, the exact
eigenvalues are unknown, so we take highly accurate computations on
very large problems as ``exact'' for these comparisons.

\begin{figure}[ht]
\begin{center}
\includegraphics[height=2.0in]{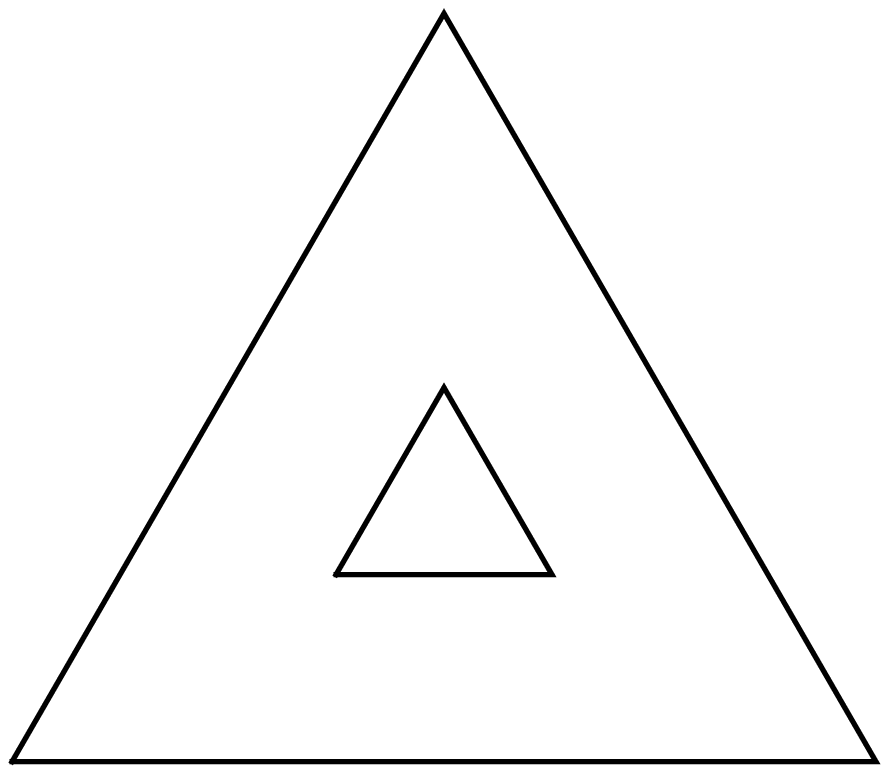}
\includegraphics[width=2.0in]{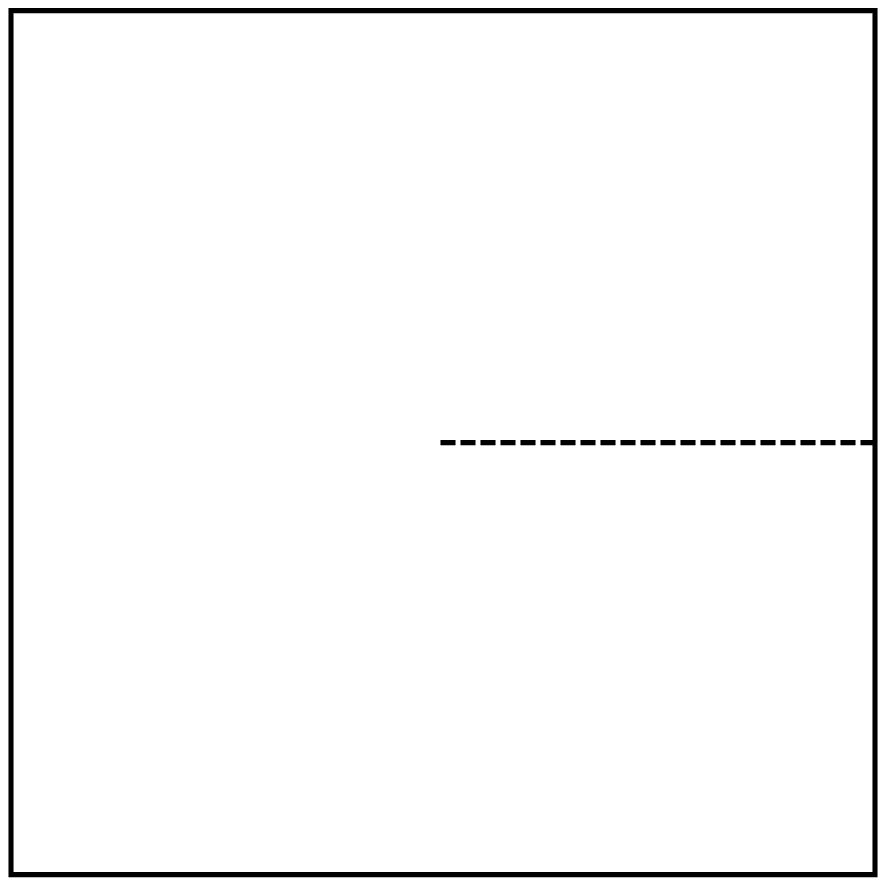}
\includegraphics[width=2.0in]{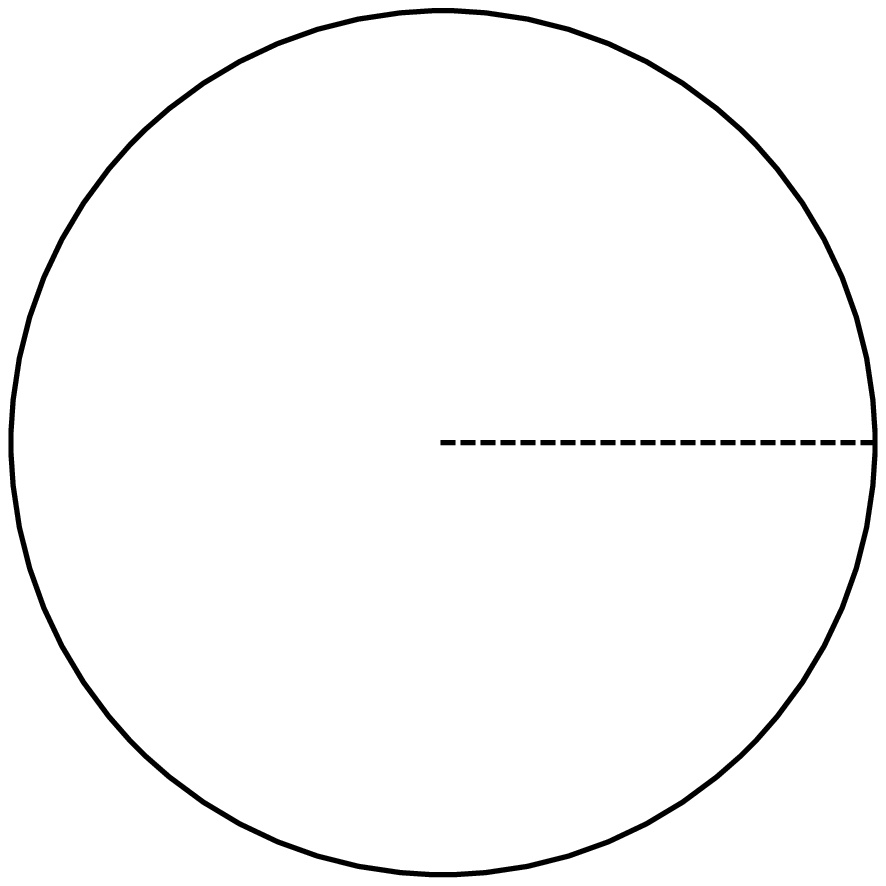}
\end{center}
\caption{\label{fig:domains} Some of the domains used for the experiments.}
\end{figure}

In all simulations we used an $hp$-adaptive algorithm in order to get the best convergence possible.
To drive the $hp$-adaptivity we use the element-wise contributions to
the quantity $\sum_{i=1}^\textsc{m}\hat\lambda_i^{-1}\varepsilon_i^2$,
to provide local error indicators. Then,
we apply a simple fixed-fraction strategy to mark the elements to
adapt. For each marked element, the choice of whether to locally
refine it or vary its approximation order is made by estimating the
local analyticity of the computed eigenfunctions in the interior of the
element by computing the coefficients of the $L^2$-orthogonal polynomial
expansion (cf.~\cite{melenk}).

\subsection{Dirichlet Laplacian on the Unit Triangle}\label{Triangle}
As a simple problem for which the eigenvalues and eigenfunctions are
explicitly known (cf.~\cite{McCartin2003}), we consider the problem
where: $\cL=-\Delta$, $\Omega$ is equilateral triangle of having
unit edge-length, and $\psi=0$ on $\partial\Omega$. The eigenvalues
can be indexed as
\[
\lambda_{mn}=\frac{16\pi^2}{9}(m^2+mn+n^2)~,
\]
and we refer interested readers to~\cite{McCartin2003} for explicit
descriptions of the eigenfunctions.
\begin{figure}[ht]
\begin{center}
\includegraphics[scale=0.6]{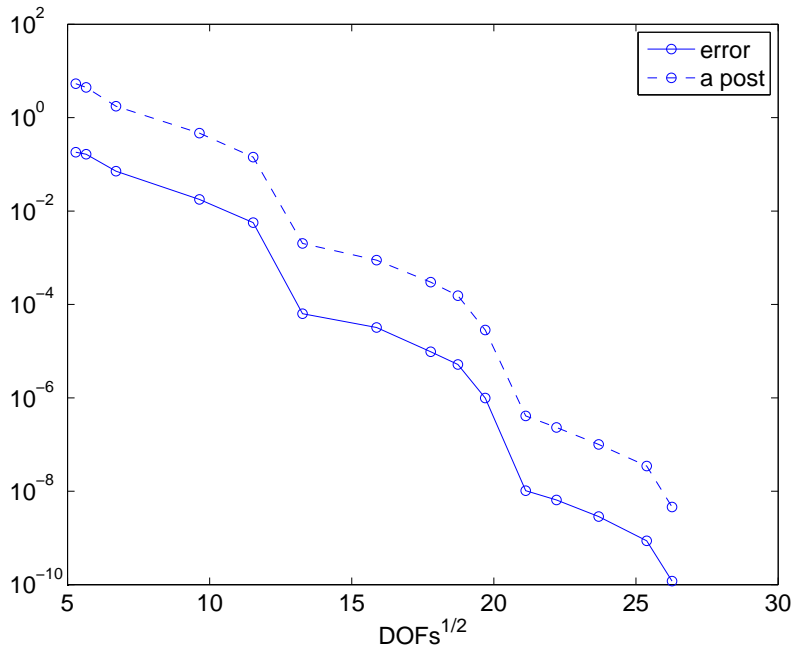}
\end{center}
\caption{\label{fig:error_tri} Errors and error estimates. Triangle problem.}
\end{figure}

\begin{figure}[ht]
\begin{center}
\includegraphics[scale=0.6]{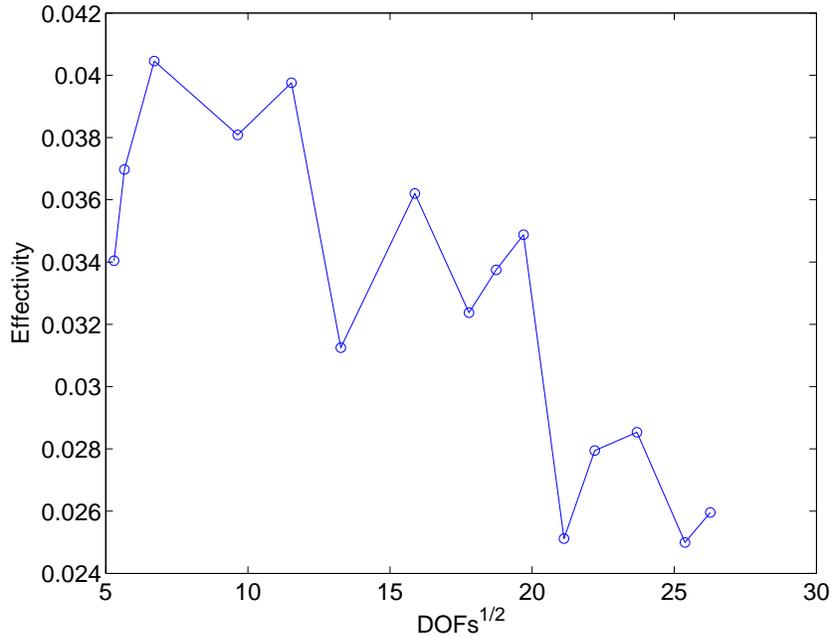}
\end{center}
\caption{\label{fig:rel_tri} Effectivity index. Triangle problem.}
\end{figure}

In Figure~\ref{fig:error_tri} we plot the total relative error for the
first four eigenvalues, together with the associated error estimate;
and in Figure~\ref{fig:rel_tri} we plot the effectivity quotient.  It
is clear that the convergence is exponential in this case, and that
the effectivity undergoes a mild degradation as the problem size
increases.  This modest decrease in effectivity is in line with
Remark~\ref{degradation}, and it is also seen in our remaining
experiments.

\subsection{Dirichlet Laplacian on the Unit Triangle with on a
  Hole}\label{Triangle_Hole}

\begin{figure}[ht]
\begin{center}
\includegraphics[scale=0.6]{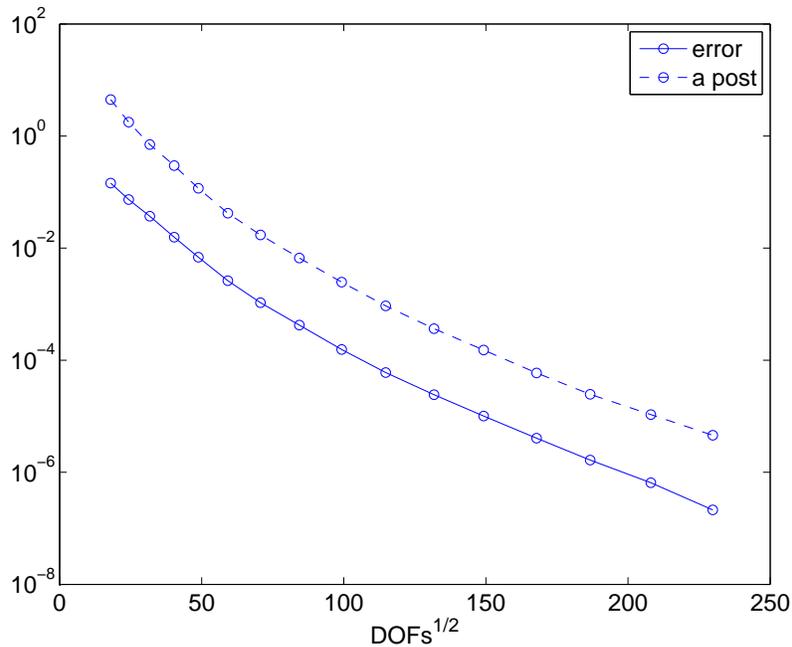}
\end{center}
\caption{\label{fig:error_triup} Errors and error estimates. Triangle
  with hole.}
\end{figure}

\begin{figure}[ht]
\begin{center}
\includegraphics[scale=0.6]{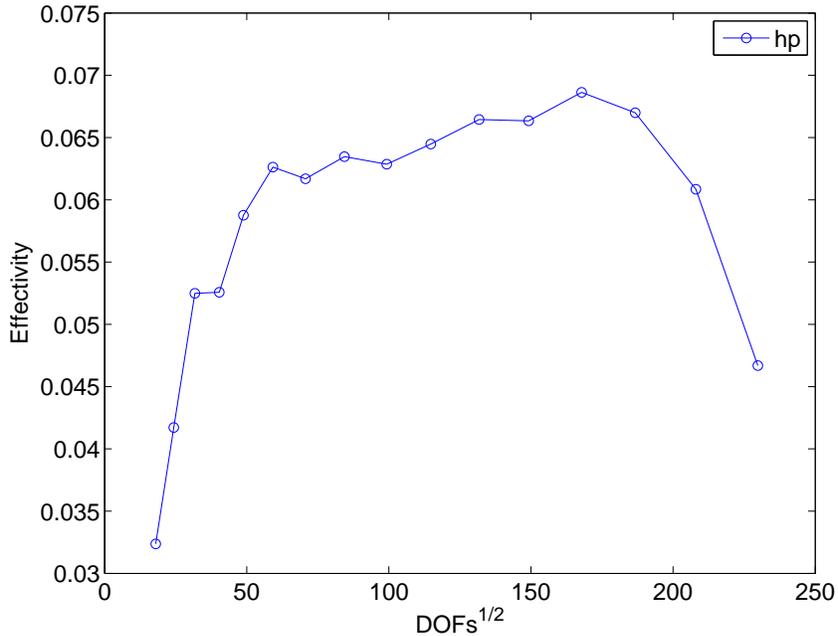}
\end{center}
\caption{\label{fig:rel_triup} Effectivity index. Triangle with hole.}
\end{figure}
We now consider the problem where $\cL=-\Delta$, $\Omega$ is the
equilateral triangle having edge-length $2$ with an equilateral triangle having
edge-length $1/2$ removed from its center (see
Figure~\ref{fig:domains}), and $\psi=0$ on $\partial\Omega$.  For such
a problem, it is expected that some of the eigenfunctions will have an
$r^{3/5}$-type singularity at each of the three interior corners,
where $r$ is the distance to the nearest corner.  In this case, the
exact eigenvalues are unknown, so we computed the following reference
values of them on a very large problem: 40.4650426 for the first
eigenvalue and 43.4868466 for the second and third, which form a
double eigenvalue. These values are accurate at least up to 1e-6.

In Figure~\ref{fig:error_triup} we plot the relative error and error
estimates together, for the first three eigenvalues, and in
Figures~\ref{fig:rel_triup} we plot the corresponding values of the
effectivity quotient.  We again see exponential convergence and a
modest deterioration of effectivity.

\subsection{Square Domain with Discontinuous Reaction Term}\label{Dauge1}
For this pair of problems we take $\Omega=(0,1)^2$,
$\nabla\psi\cdot\mb{n}=0$ on $\partial\Omega$, and
$\cL\,\psi=-\Delta \psi+\kappa V_{MD}\cdot\psi$, where $V_{MD}$ is the
characteristic function of the touching squares labelled $\cM_1$ in
Figure~\ref{figMonique}.  We consider two values of the constant
parameter, $\kappa= 10, 100$.  It is straightforward to see that the
corresponding bilinear form is an inner-product in this case (no zero
eigenvalues), and that all eigenfunctions are at least in $H^2$.

\begin{figure}[ht]
\begin{center}
\begin{picture}(0,0)%
\includegraphics{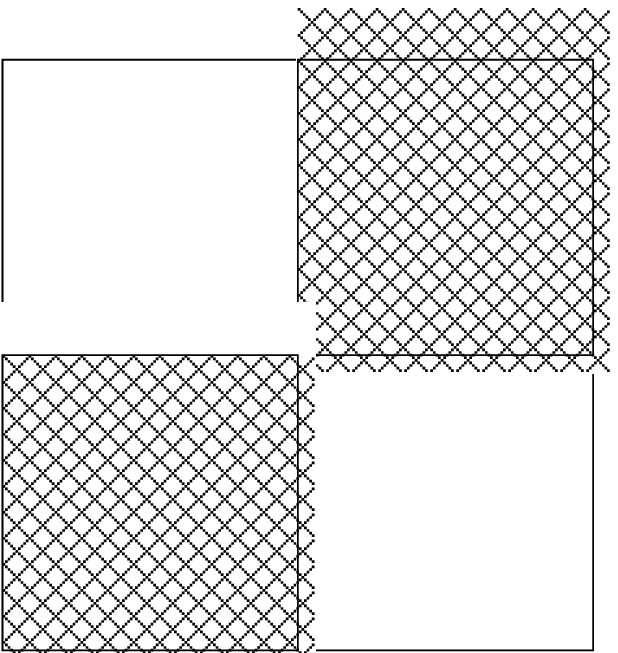}%
\end{picture}%
\setlength{\unitlength}{4144sp}%
\begingroup\makeatletter\ifx\SetFigFont\undefined%
\gdef\SetFigFont#1#2#3#4#5{%
  \reset@font\fontsize{#1}{#2pt}%
  \fontfamily{#3}\fontseries{#4}\fontshape{#5}%
  \selectfont}%
\fi\endgroup%
\begin{picture}(2724,2724)(1339,-523)
\put(1850,1490){\makebox(0,0)[lb]{\smash{{\SetFigFont{12}{14.4}{\familydefault}{\mddefault}{\updefault}{$\mathcal{M}_2$}%
}}}}
\put(3200,140){\makebox(0,0)[lb]{\smash{{\SetFigFont{12}{14.4}{\familydefault}{\mddefault}{\updefault}{$\mathcal{M}_2$}%
}}}}
\put(3200,1490){\makebox(0,0)[lb]{\smash{{\SetFigFont{12}{14.4}{\familydefault}{\mddefault}{\updefault}{$\mathcal{M}_1$}%
}}}}
\put(1850,140){\makebox(0,0)[lb]{\smash{{\SetFigFont{12}{14.4}{\familydefault}{\mddefault}{\updefault}{$\mathcal{M}_1$}%
}}}}
\end{picture}%
\end{center}
\caption{\label{figMonique} A modification of the touching squares
  example of M. Dauge.}
\end{figure}


For $\kappa=10$, we have in Figure~\ref{fig:error_neu_potential_10_0}
the total relative error and error estimates for the first four
eigenvalues; and the effectivity quotient is given in
Figure~\ref{fig:rel_neu_potential_10_0}. For these simulations we used
the following reference values for the first four eigenvalues, which
are 1e-8 accurate: 4.150242455, 10.706070962, 18.779725462,
25.150325247.  The analogous plots for the first four eigenvalues in
the case $\kappa=100$ are given in
Figure~\ref{fig:error_neu_potential_100_0} and
Figure~\ref{fig:rel_neu_potential_100_0}.  For these simulations, we
used the following reference values for the first four eigenvalues,
which are 1e-8 accurate: 13.210576406, 13.990033964, 60.294151672,
64.840268299.  In both cases we see apparent exponential convergence,
and reasonable effectivity behavior.  It seem clear from the error
plots that for both values of $\kappa$ the convergence is exponential.

\begin{figure}[ht]
\begin{center}
\includegraphics[scale=0.6]{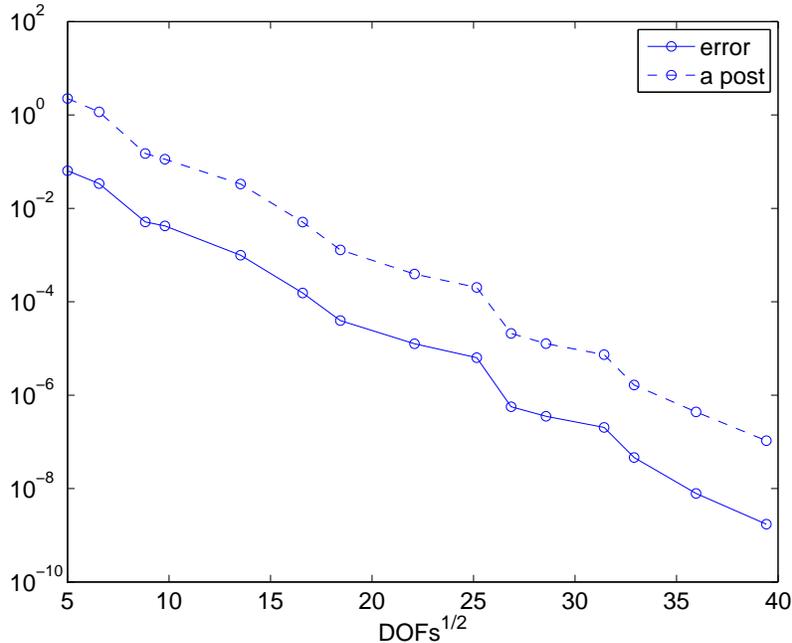}
\end{center}
\caption{\label{fig:error_neu_potential_10_0} Errors and error estimates.
  Discontinuous reaction term, $\kappa=10$.}
\end{figure}

\begin{figure}[ht]
\begin{center}
\includegraphics[scale=0.6]{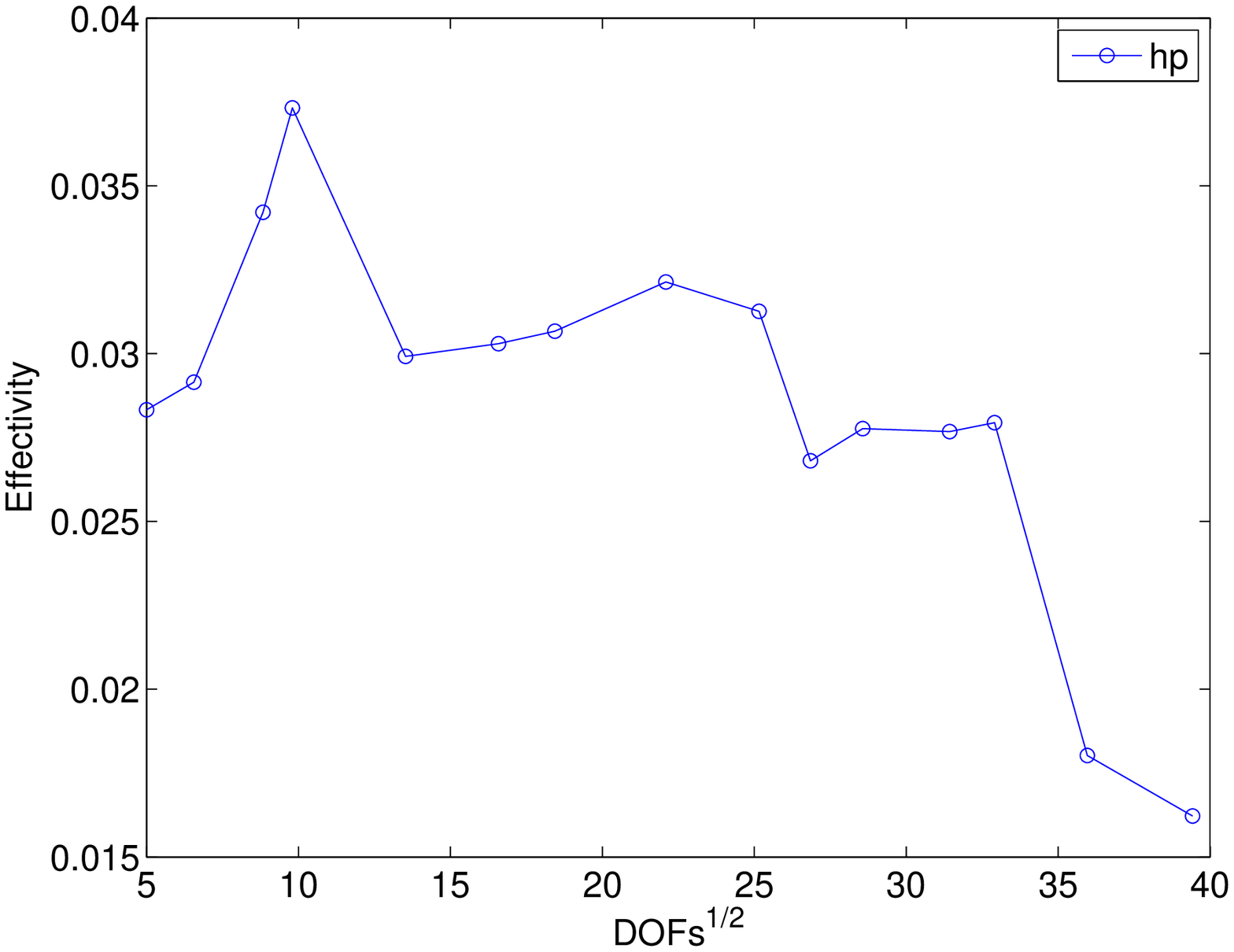}
\end{center}
\caption{\label{fig:rel_neu_potential_10_0} Effectivity
  index. Discontinuous reaction term, $\kappa=10$.}
\end{figure}

\begin{figure}[ht]
\begin{center}
\includegraphics[scale=0.6]{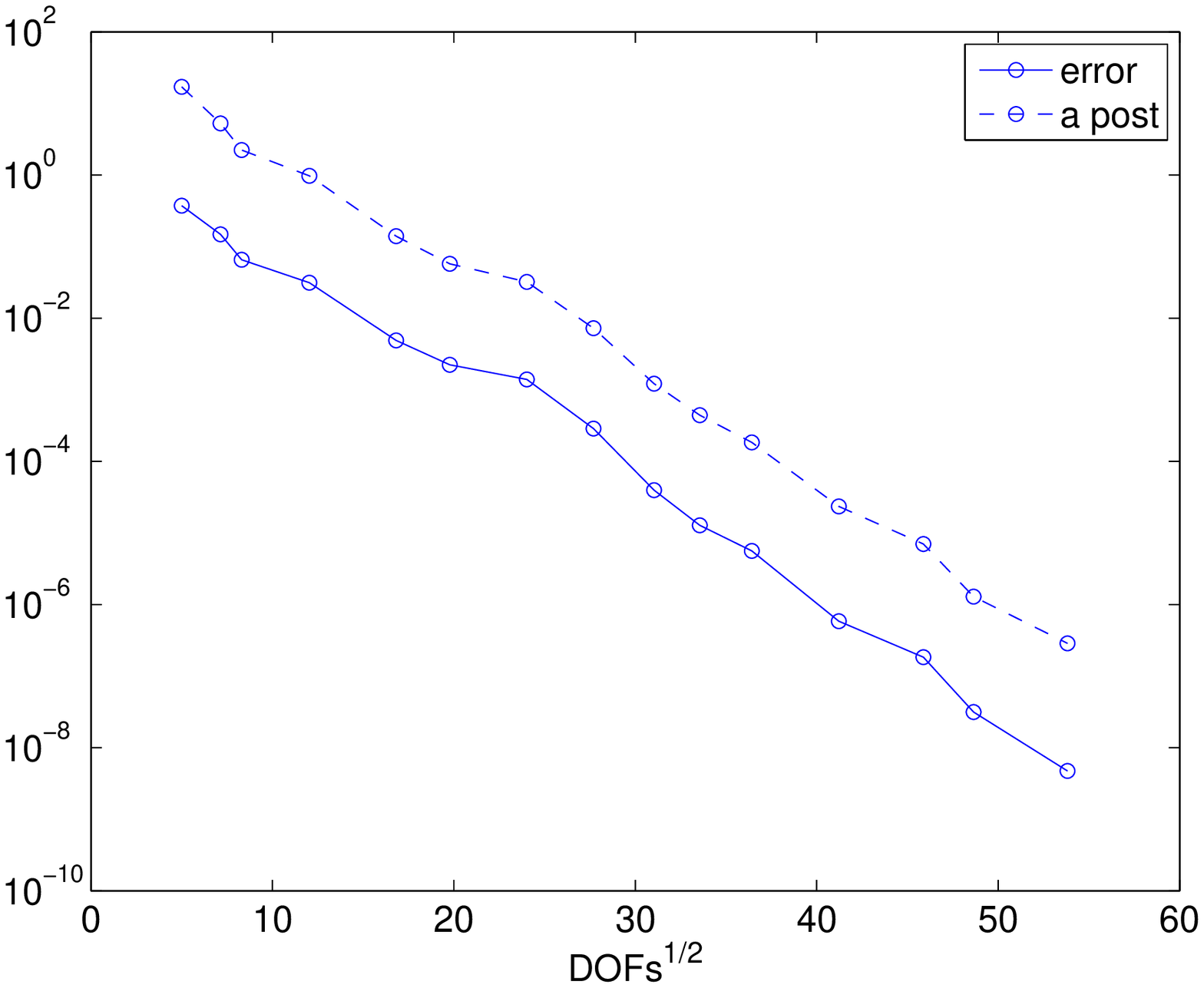}
\end{center}
\caption{\label{fig:error_neu_potential_100_0} Errors and error estimates.
  Discontinuous reaction term, $\kappa=100$.}
\end{figure}

\begin{figure}[ht]
\begin{center}
\includegraphics[scale=0.6]{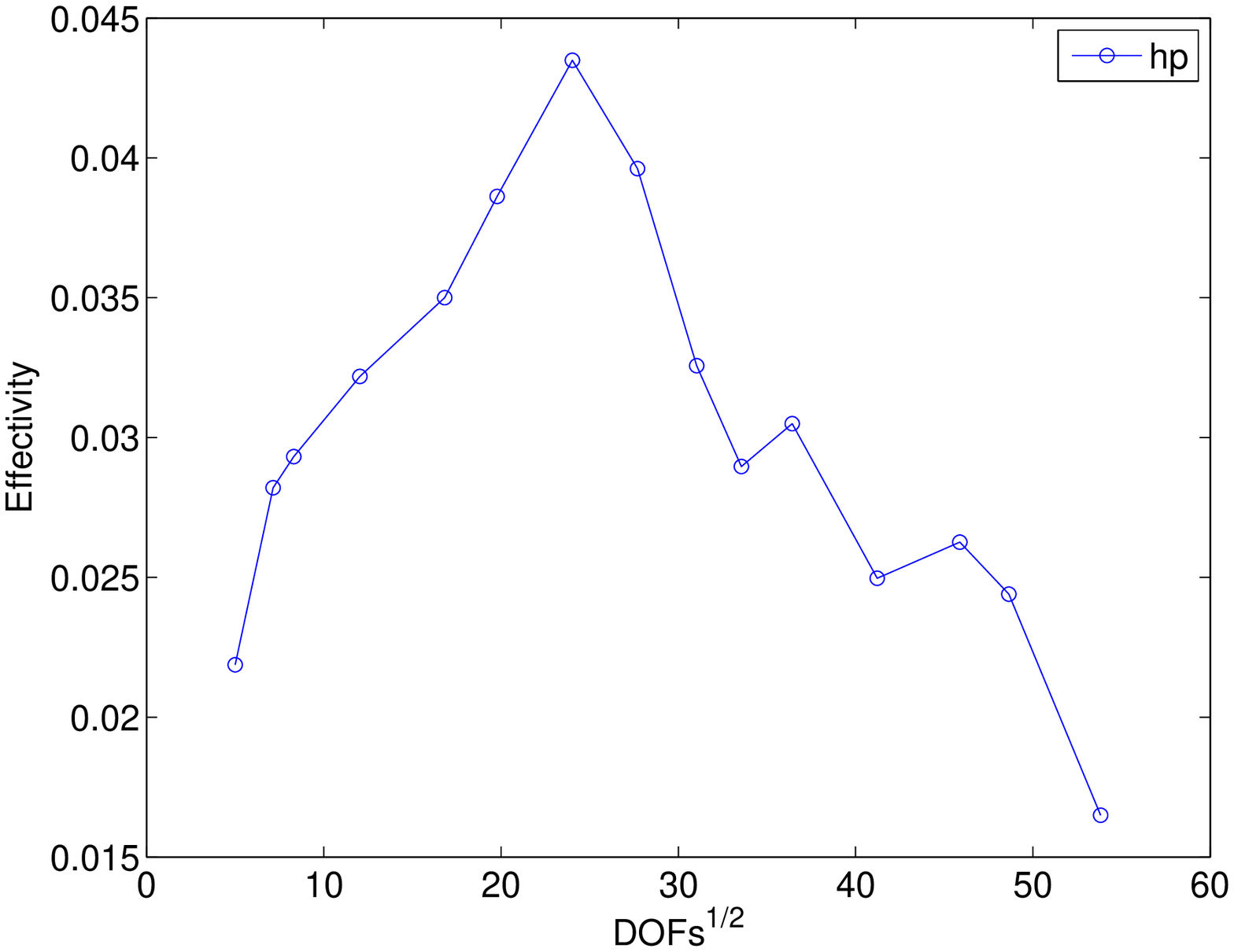}
\end{center}
\caption{\label{fig:rel_neu_potential_100_0} Effectivity
  index. Discontinuous reaction term, $\kappa=100$.}
\end{figure}

\subsection{Square Domain with Discontinuous Diffusion Term}
Using the domain $\Omega=(0,1)^2$, partitioned into regions $\cM_1$
and $\cM_2$ as in Figure~\ref{figMonique}, and homogeneous Dirichlet
conditions $\psi=0$ on $\partial\Omega$, we consider the operator
$\cL=-\nabla\cdot(a\nabla)$, where $a=1$ in $\cM_2$ and $a$ in $\cM_1$
may vary.  Such problems can have arbitrarily bad singularities
at the cross-point of the domain depending on the relative sizes of $a$ in the two
subdomains---see, for example,~\cite{Kellogg1975,Knyazev2003}
and~\cite[Example 5.3]{Morin2000}.

We have considered two values for $a$ in $\cM_1$: 10 and 100. Since
the exact eigenvalues are not available, we computed the following
three reference values for the first three eigenvalues when $a=10$:
64.226529416, 75.028156269, 141.161506328; and the following three
reference values for the first three eigenvalues when $a=100$:
77.800981966, 78.564198245, 193.916538067. All reference values are at
least 1e-8 accurate.  The relative error and effectivity plots for
both cases are given in
Figures~\ref{fig:error_jumps_dauge_10}-\ref{fig:rel_jumps_dauge_100},
and again we see apparent exponential convergence.

\begin{figure}[ht]
\begin{center}
\includegraphics[scale=0.6]{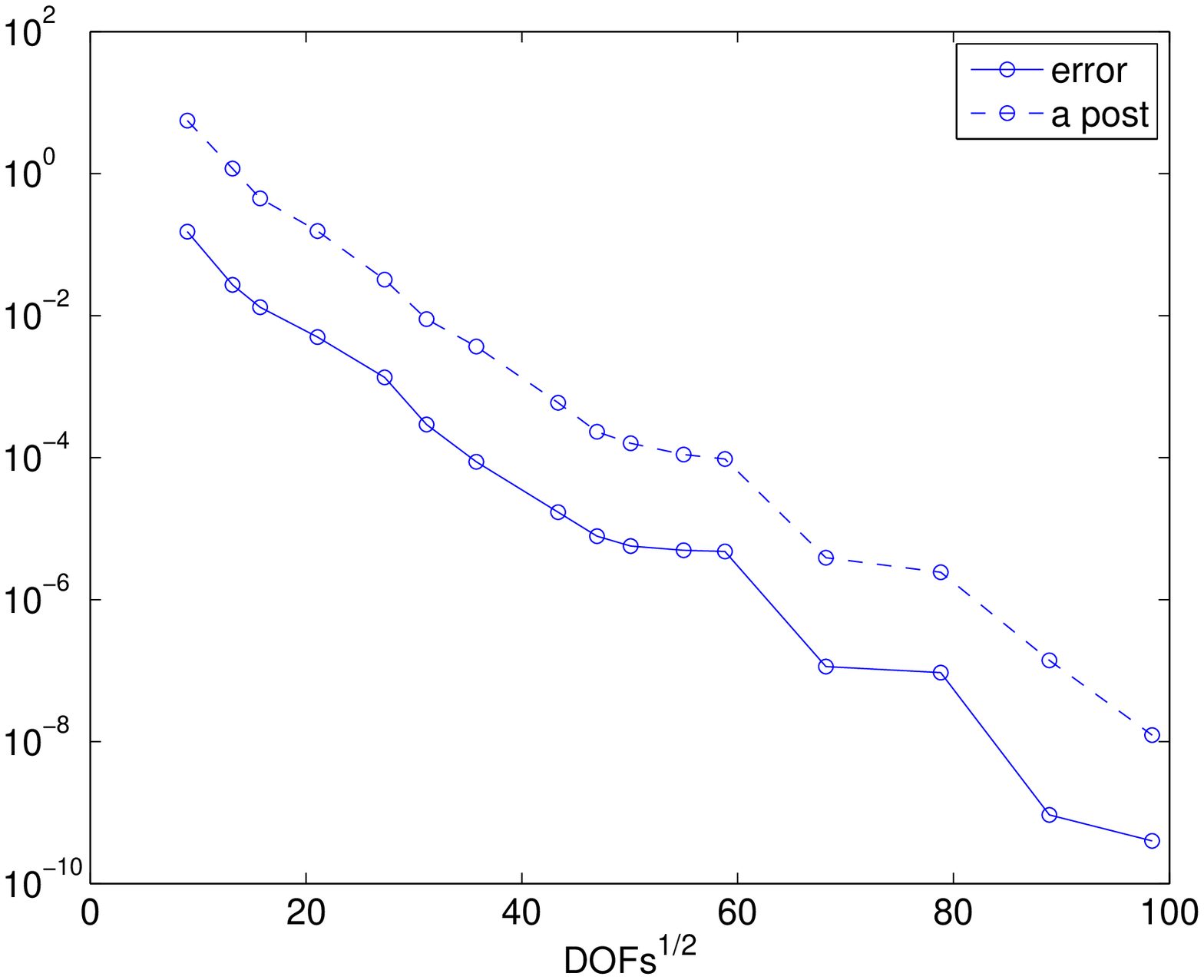}
\end{center}
\caption{\label{fig:error_jumps_dauge_10} Errors and error estimates.
  Discontinuous diffusion term, $\kappa=10$.}
\end{figure}

\begin{figure}[ht]
\begin{center}
\includegraphics[scale=0.6]{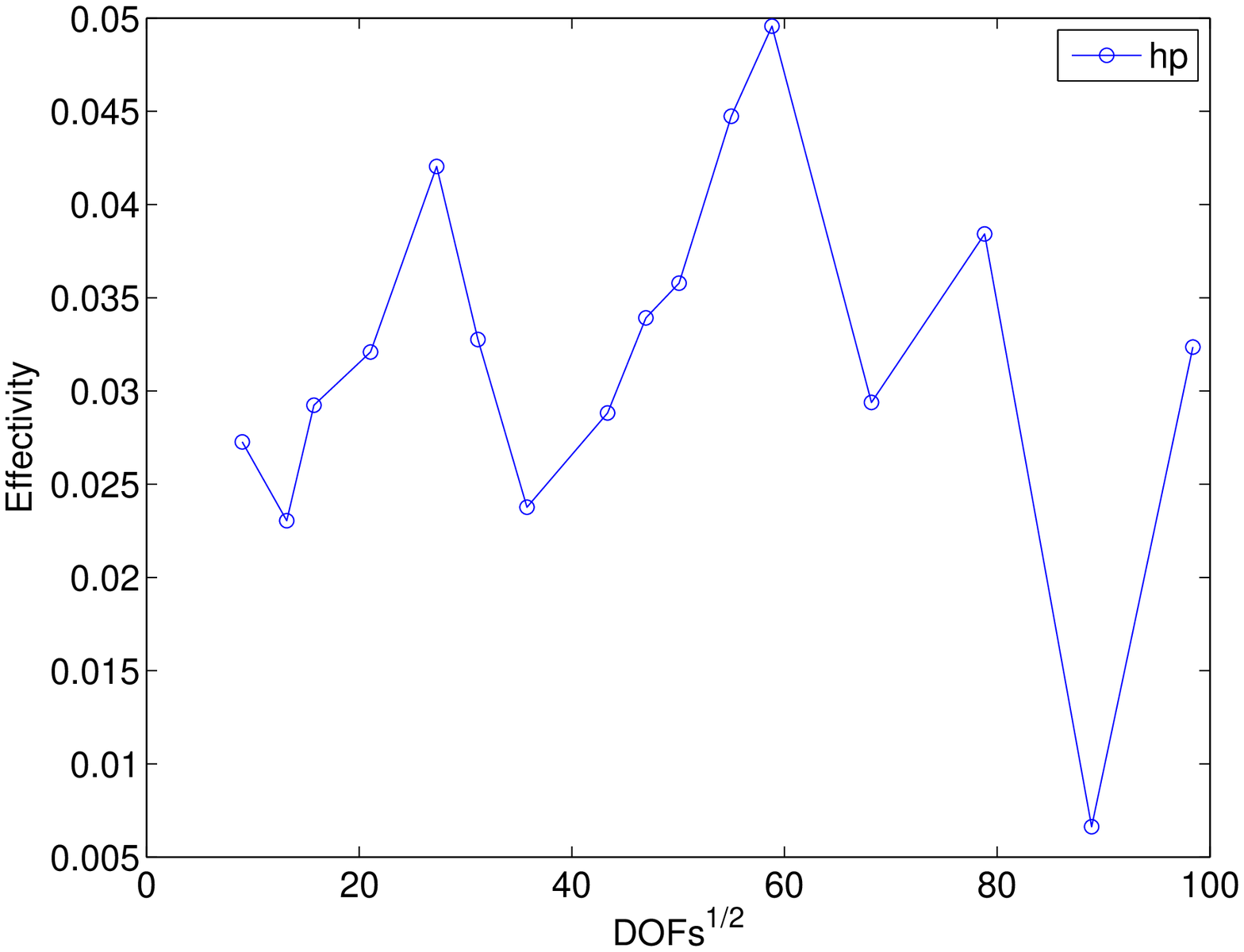}
\end{center}
\caption{\label{fig:rel_jumps_dauge_10} Effectivity
  index. Discontinuous diffusion term, $\kappa=10$.}
\end{figure}

\begin{figure}[ht]
\begin{center}
\includegraphics[scale=0.6]{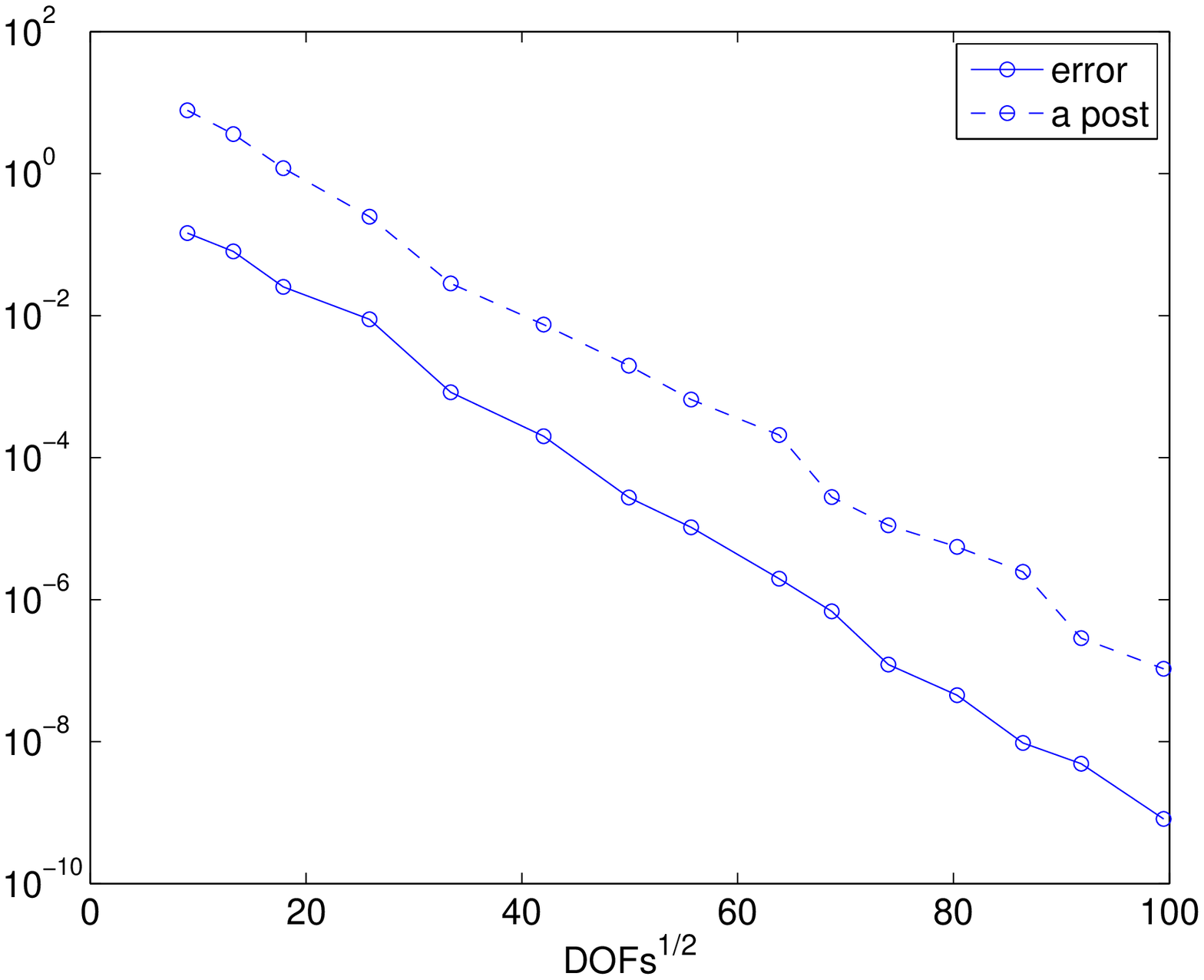}
\end{center}
\caption{\label{fig:error_jumps_dauge_100} Errors and error estimates.
  Discontinuous diffusion term, $\kappa=100$.}
\end{figure}

\begin{figure}[ht]
\begin{center}
\includegraphics[scale=0.6]{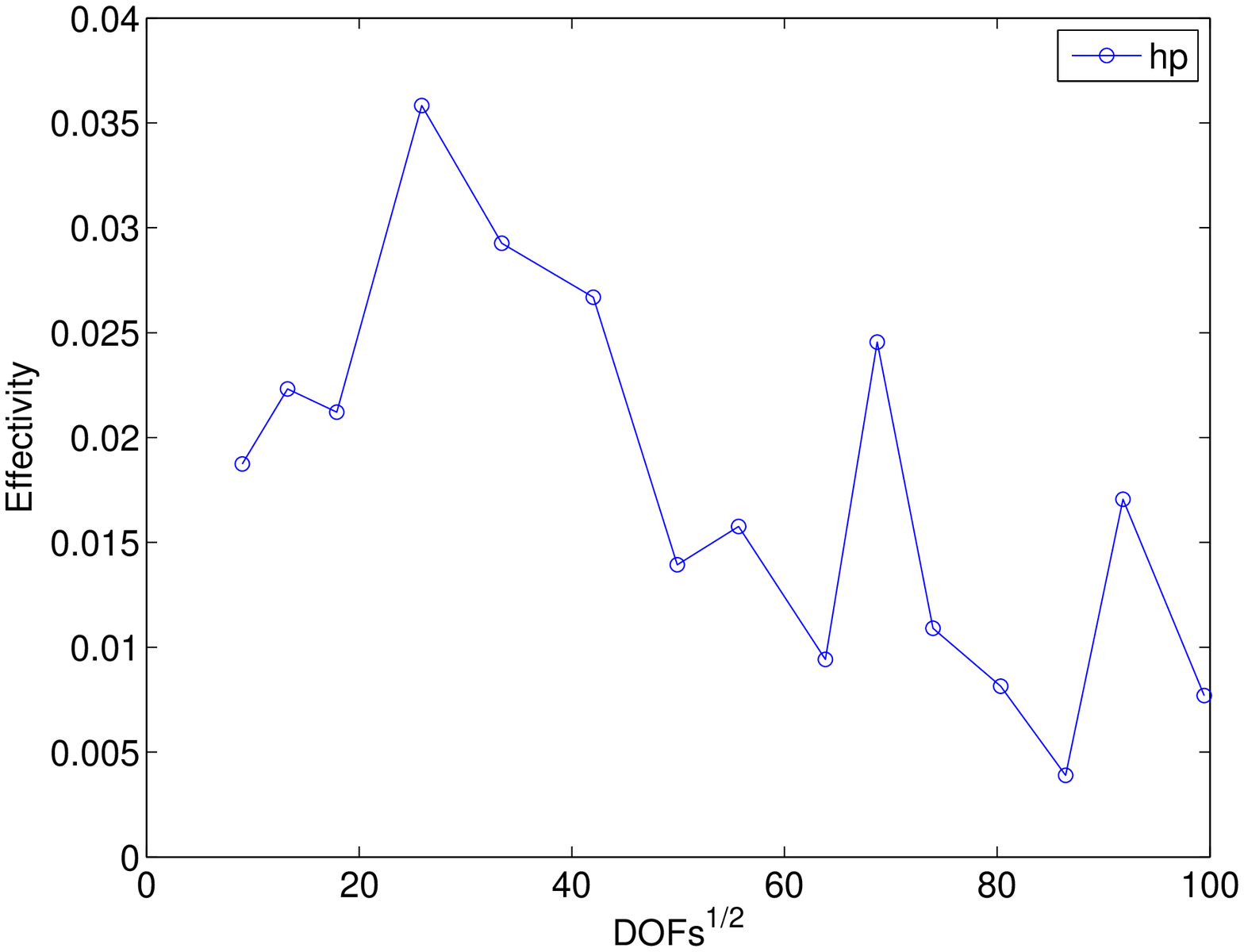}
\end{center}
\caption{\label{fig:rel_jumps_dauge_100} Effectivity
  index. Discontinuous diffusion term, $\kappa=100$.}
\end{figure}

\subsection{Square Domain with a Slit}
For this problem, $\cL=-\Delta$ and $\Omega=(0,1)^2\setminus S$, where
$S=\{(x,1/2):\,1/2<x<1\}$; this is pictured in
Figure~\ref{fig:domains}, with $S$ as the dashed segment.  Homogeneous
Neumann conditions are imposed on both ``sides'' of $S$ and
homogeneous Dirichlet boundary conditions are imposed on the rest of
the boundary of $\Omega$.  For this example we used the following
reference values for the first four eigenvalues with in brackets the
corresponding accuracy:
20.739208802(8), 34.485320(5), 50.348022005(8), 67.581165196(8).

To give some indication of the nature of the eigenfunctions in the
interior, we briefly consider a related problem where $\Omega$ is the
unit disk with a slit along the positive $x$-axis, as pictured in
Figure~\ref{fig:domains}, with the same boundary conditions. In this
case, the eigenvalues and eigenfunctions are known explicitly.  For
$k\geq 0$ and $m\geq 1$, let $z_{km}$ be the $m^{th}$ positive root of
the first-kind Bessel function $J_{k/2}$.  It is straightforward to
verify that, up to renormalization of eigenfunctions, the eigenpairs
can be indexed by
\begin{align*}
\lambda_{km}=z_{km}^2\quad,\quad
\psi_{km}=J_{k/2}(z_{km}r)\sin(k\theta/2)\quad,\quad k,m\in\NN~.
\end{align*}
We see that $\psi_{km}\sim \sin(k
\theta/2)\left(\frac{z_{km}r}{2}\right)^{1/2}$ as $r\to 0$, so
singularities of types $r^{k/2}$ occur infinitely many times in the
spectrum.  The strongest of these singularities is of type $r^{1/2}$,
and it occurs in the eigenfunction associated with the second
eigenvalue, for example.  The same asymptotic behavior of the
eigenfunctions near the crack tip is expected for the square
and circular domains, and in Figure~\ref{fig:eig2} we show
contour plots of the second eigenvalue in both cases.  For the
circular domain, the second eigenvalue and function can, in
principle be obtained to arbitrary precision using a computer
algebra system.  Using {\sc Mathematica}, we computed the
second eigenvalue (the smallest positive root of $J_{1/2}$)
to 20 digits, 9.869604401089358619, and generated the
corresponding contour plot of the eigenfunction.
\begin{figure}[ht]
\begin{center}
\includegraphics[width=3.0in]{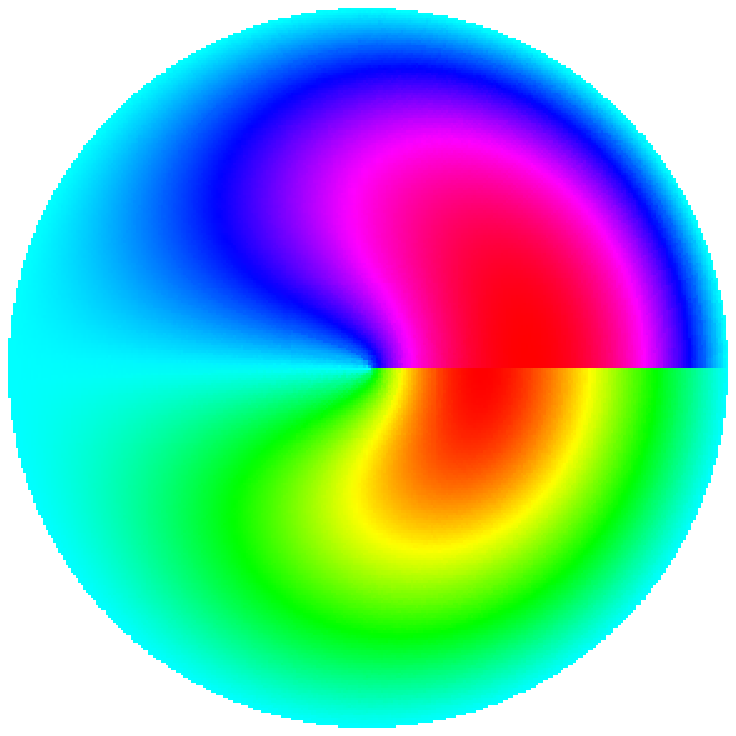}
\includegraphics[width=3.0in]{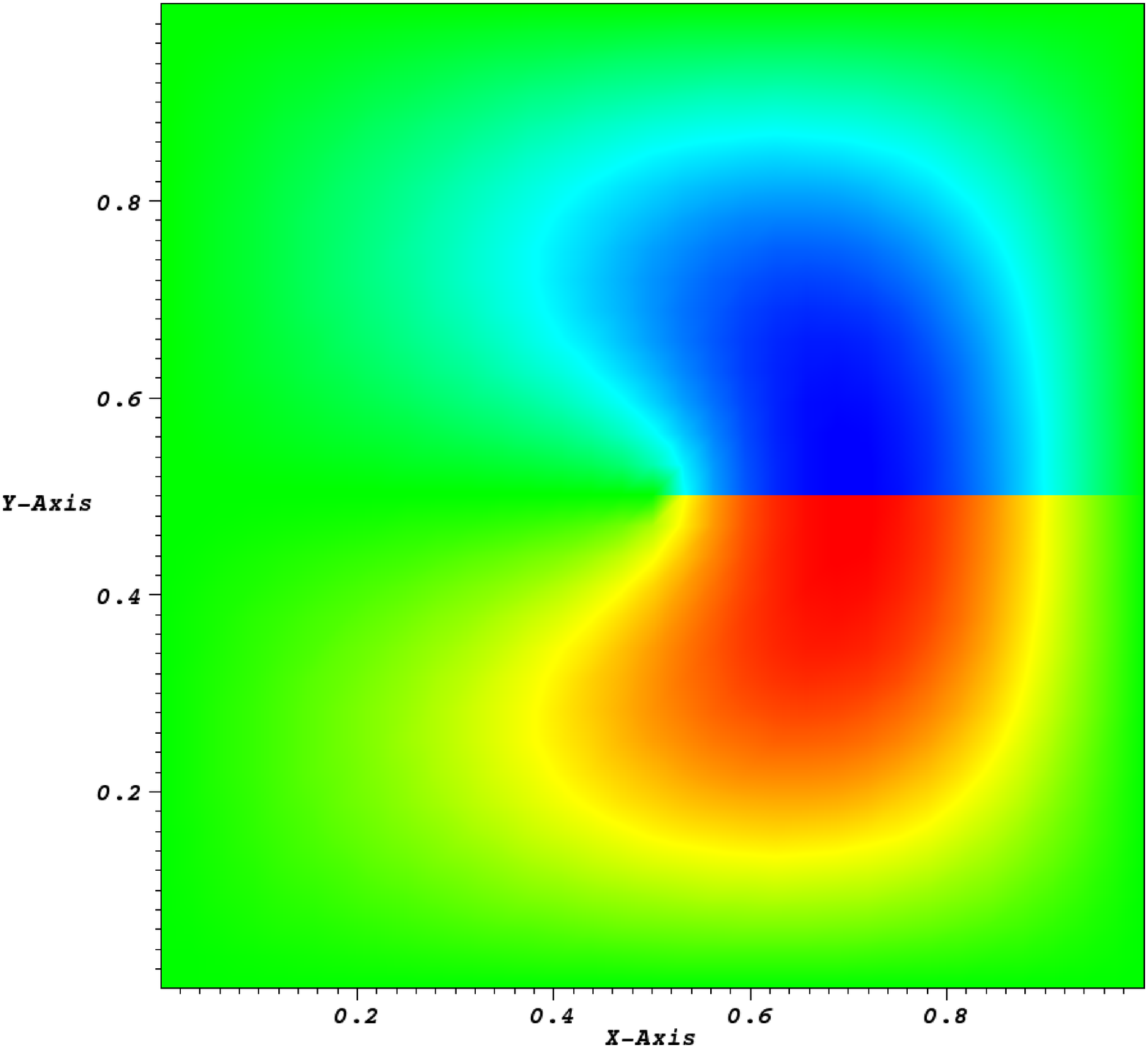}
\end{center}
\caption{\label{fig:eig2} Second eigenfunction for the slit circle
  (top) and slit square.}
\end{figure}

\begin{figure}[ht]
\begin{center}
\includegraphics[scale=0.6]{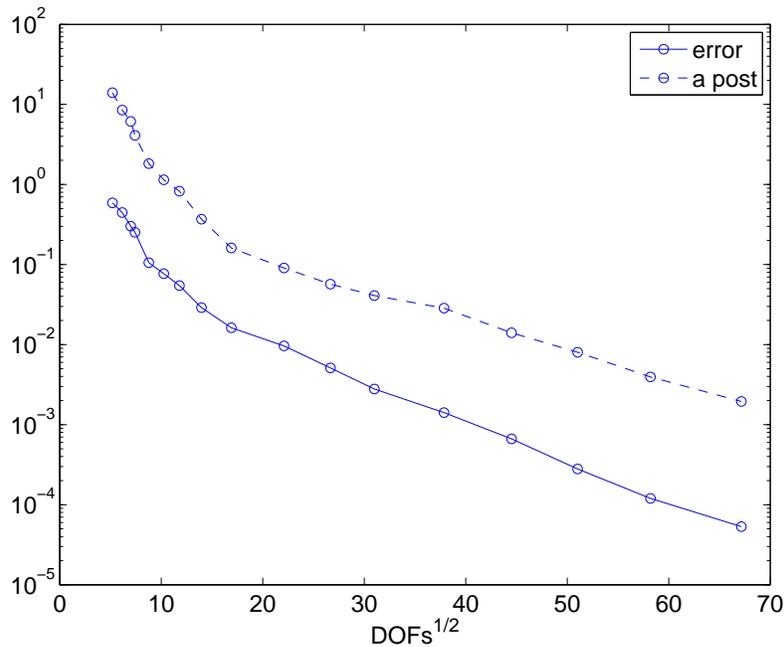}
\end{center}
\caption{\label{fig:lap_err_crack_hp} Errors and error estimates. Slit
  square. First four eigenvalues.}
\end{figure}

\begin{figure}[ht]
\begin{center}
\includegraphics[scale=0.6]{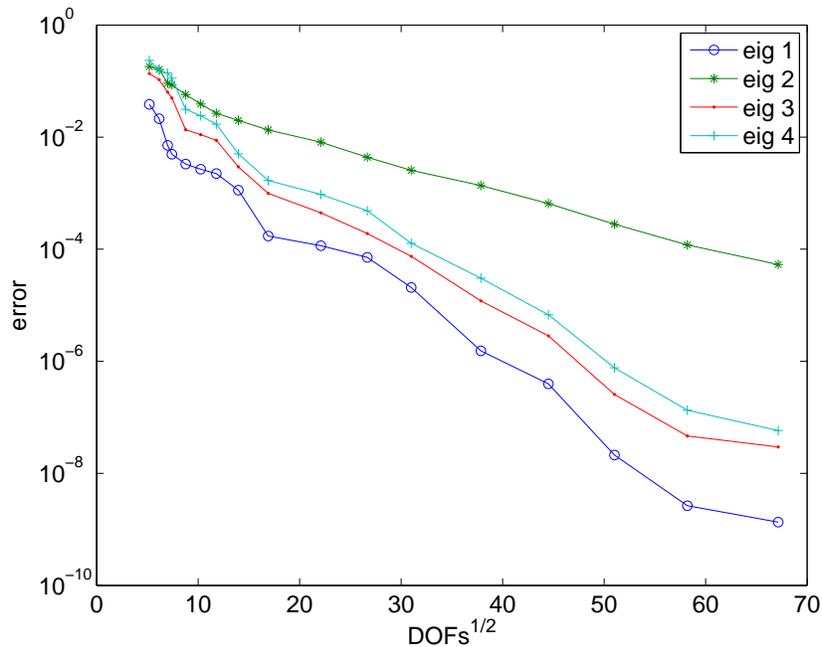}
\end{center}
\caption{\label{fig:lap_err_split_crack_hp} Errors from each eigenvlaue. Slit
  square.}
\end{figure}

\begin{figure}[ht]
\begin{center}
\includegraphics[scale=0.6]{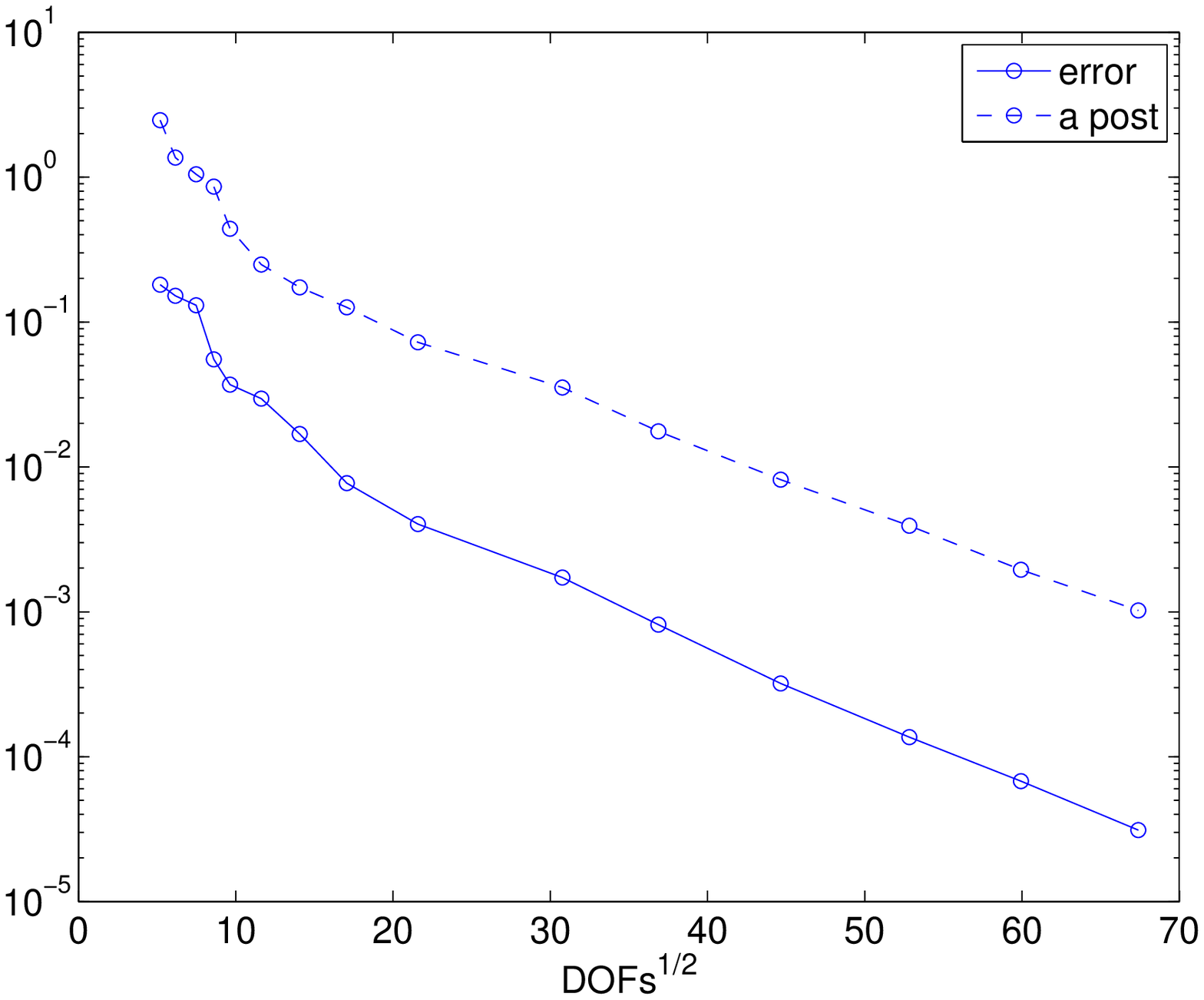}
\end{center}
\caption{\label{fig:lap_err_crack_hp2} Errors and error estimates. Slit
  square. Second eigenvalue only.}
\end{figure}

In Figure~\ref{fig:lap_err_crack_hp} we plot the total relative errors
and error estimates for the first four eigenvalues, and in
Figure~\ref{fig:lap_err_split_crack_hp} the individual eigenvalue
errors are shown.  It is clear from the second of these figures that
the second, which corresponds to the most singular eigenfunction,
clearly has the worst convergence rate (as expected), and that this is
what ``spoils'' the convergence of the cluster of the first four
eigenvalues.  This becomes even more apparent when
Figure~\ref{fig:lap_err_crack_hp2}, which corresponds to the second
eigenvalue alone, is compared with
Figure~\ref{fig:lap_err_crack_hp}---they are nearly identical.

\section*{Acknowledgement}
L. G. was supported by the grant:
``Spectral decompositions -- numerical methods and applications'', Grant Nr. 037-0372783-2750 of the Croatian MZOS.
We would like to thanks Paul Houston and Edward Hall for kind support and very useful discussions.
\bibliographystyle{abbrv}
\def\cprime{$'$}

\end{document}